%% file: main.tex
\documentclass{article}
\usepackage[utf8]{inputenc}
\usepackage{subfiles}
\input{preamble.tex}
\usepackage{graphicx}
\usepackage{wrapfig}
\usepackage[maxbibnames=6, minbibnames=6]{biblatex}
\usepackage{hyperref}
\title{Hamiltonian Braids via Generating Functions}
\author{Francesco Morabito}
\date{October 2024}
\begin{document}
\maketitle
\begin{abstract}
    Given a compactly supported Hamiltonian diffeomorphism of the plane, one can define a generating function for it. In this paper, we show how generating functions retain information about the braid type of collections of fixed points of Hamiltonian diffeomorphisms. One the one hand, we show that it is possible to define a filtration keeping track of linking numbers of pairs of  fixed points on the Morse complex of the generating function. On the other, we provide a finite-dimensional proof of a Theorem by Alves and Meiwes about the lower-semicontinuity of the topological entropy with respect to the Hofer norm. The technical tools come from work by Le Calvez which was developed in the 90s. In particular, we apply a version of positivity of intersections for generating functions.
\end{abstract}
\tableofcontents
\section{Introduction}
\subsection{Overview}

A symplectic manifold $(M, \omega)$ is a smooth manifold $M$ equipped with a smooth, closed and non-degenerate de Rham 2-form $\omega$. Given a (possibly) time-dependent function $H\in\mathscr{C}^\infty(\sphere^1_t\times M; \R)$ on the manifold $M$, we can define the associated Hamiltonian vector  field $X_H$ by the equation
\begin{equation*}
-dH_t=\iota_{X_{H_t}}\omega.
\end{equation*}
The function $H$ is then called a Hamiltonian. The time 1-map of $X_H$, when defined, is called a Hamiltonian diffeomorphisms. We denote it by $\phi^1_H$. If the manifold is closed the time 1-maps are always defined, and form the group of Hamiltonian diffeomorphisms of $M$, denoted $\Ham(M,\omega)$. If $M$ is non-compact, one may restrict their attention to compactly supported Hamiltonians: their time 1-maps are always well defined and they form the group of compactly supported Hamiltonian diffeomorphisms of $M$, denoted $\Ham_c(M, \omega)$. Of course, whenever $M$ is closed, $\Ham_c(M, \omega)=\Ham(M, \omega)$. In the following we shall omit the symplectic form from the notation, as it will be either obvious or irrelevant.

The study of the dynamics of Hamiltonian diffeomorphisms is a flourishing domain of research, and itself one of the motivation for the existence of Symplectic Topology. Several tools have been developed to study it. In this paper we focus our attention on Generating Functions. They encode in a Morse-theoretical finite-dimensional problem the information about the fixed points of Hamiltonian diffeomorphisms, together with a real number called ``symplectic action'' of the fixed points. One can construct a filtered chain complex from a generating function. We shall restrict our attention on Hamiltonian diffeomorphisms of the plane. A generating function is, roughly speaking, a function defined on a vector bundle on the plane, and whose critical points are in bijection with the fixed points of the Hamiltonian diffeomorphism we are looking at. If said Hamiltonian diffeomorphism is non-degenerate (a generic condition), the generating function turns out to be Morse, and we can construct its Morse complex. The value of the Morse function at a critical point is called the ``action'' of the associated fixed point of the diffeomorphism, and the action gives rise to a tautological filtration on the Morse complex, the famous action filtration. This filtration which may in turn be used to prove results, for instance, about the Hofer metric.

\begin{defn}
Let $\varphi\in \Ham_c(M)$ be a Hamiltonian diffeomorphism. Its Hofer norm is defined to be\begin{equation*}
    \norm{\varphi}:=\inf_{H\vert \phi^1_H=\phi}\int_0^1 \left\{\max_{x\in M}H_t(x)-\min_{x\in M}H_t(x)\right\}dt.  
\end{equation*}  
\end{defn}

In this paper we show how one can encode further information in the Morse complex of a generating function. In particular, we shall associate to a collection of $k$ distinct critical points a braid, defined to be a collection of Hamiltonian loops based at the associated fixed points of the diffeomorphism. When $k=2$, braids are entirely classified by their linking number. We prove the following Theorem.

\begin{thm}\label{thm:A}
    Let $\varphi$ be a compactly supported Hamiltonian diffeomorphism of the plane with its standard symplectic form $dx\wedge dy$. Assume $\varphi$ is non degenerate on the interior of its support. Then for any generating function quadratic at infinity $S: \R^2\times \R^k\rightarrow \R$ for $\varphi$, there exists a non degenerate quadratic form $Q$ on $\R^l$, a Riemannian metric $g$ on $\R^{2+k+l}$ such that the pair $(S\oplus Q, g)$ is both Palais-Smale and Morse-Smale, and making (an extension of) the function
    \begin{align*}
        I: CM(&S\oplus Q, g; \Z)\otimes CM(S\oplus Q, g; \Z)\rightarrow \Z,\\ &I(x\otimes y):=\begin{cases}
            \frac{1}{2}\mathrm{lk}(x, y) & x\neq y\\
            -\ceil*{\frac{CZ(x)}{2}} & x=y
        \end{cases}
    \end{align*}
    into an increasing filtration of the tensor complex. Here, $x$ and $y$ are both critical points of $S\oplus Q$, and their associated fixed points of $\varphi$.
\end{thm}

The tools we use in the proof are what we call ``twist generating functions'', whose construction may for instance be found in \cite{leCalvez1999}. These generating functions have good properties with respect to linking numbers, properties that we shall later recall. In particular, they satisfy a version of the positivity of intersections for pseudo-holomorphic curves in dimension 4 (see \cite[Appendix E]{mcDSal16}). We are going to apply a Theorem of uniqueness for generating functions by Viterbo (see \cite{the99}, \cite{viterbo92}) to obtain the Theorem.

We then consider braids with an arbitrary number of strands. Studying the behaviour of continuation maps one can prove a (very slightly weaker version of a) Theorem by Alves and Meiwes (\cite[Theorem 3]{alvMei21}) using generating functions. The original proof uses instead properties of intersections between pseudoholomorphic curves. We postpone the definition of $\varepsilon$-isolated collection of orbits, the reader may assume it is just a condition on the actions of the components of the collection. A proper definition will be provided immediately before the proof of the Theorem, in Section \ref{sec:proofThmB}.

\begin{thm}\label{thm:B}
    Let $\varphi\in\mathrm{Ham}_c(\D)$ be non degenerate, and let $\{x_1,\dots, x_k\}$ be an $3\varepsilon$-isolated collection of fixed points of $\varphi$. Then there exists a positive $\delta=\delta(\varepsilon)$ such that, for every non-degenerate $\varphi'\in\mathrm{Ham}(\D)$ with $d_H(\varphi, \psi)<\delta$, there exists a collection of fixed points $\{y_1, \dots, y_k\}$ of $\varphi'$ such that
    \begin{equation*}
        b(x_1, \dots, x_k)=b(y_1, \dots, y_k).
    \end{equation*}
\end{thm}
\begin{rem}
    The roles of $\varphi$ and $\psi$ are not symmetric, since we only make the isolation assumption on the orbits of $\varphi$.
\end{rem}
\begin{rem}
Our Theorem \ref{thm:B} is weaker than Alves and Meiwes's result in the following sense: they have a better estimate on the Hofer distance the braids persist for; moreover, they prove the result on all surfaces. We think an extension via generating functions to all closed surfaces except for the sphere is feasible.
\end{rem}

Their proof, which we can translate in our language almost word for word, relies on some simple energy estimates by which we can show that the continuation map from the Morse complex of $\varphi$ to that of $\varphi'$ is a deformation retract. This fact is in fact enough to conclude.

From the previous Theorem,  Alves and Meiwes are able to prove the following result:
\begin{cor}
    The topological entropy
    \begin{equation*}
        h_{\mathrm{top}}:\mathrm{Ham}_c(\D)\rightarrow [0,+\infty]
    \end{equation*}
    is a lower semicontinuous function with respect to the Hofer metric.
\end{cor}
We shall not recall the proof here, we refer to \cite[Section 2.1]{alvMei21} for it.
\begin{rem}\label{rem:techRem}
    We need to mention here (and we thank the referee for highlighting this point) that from Theorem \ref{thm:B} one can directly only prove that non-degenerate Hamiltonian diffeomorphisms are points of lower-semicontinuity for the topological entropy. The proof of its lower-semicontinuity at any Hamiltonian diffeomorphisms hinges instead on a slight improvement of Theorem \ref{thm:B} which encompasses the case where the braids are formed by non-degenerate orbits, the Hamiltonian diffeomorphism not being assumed to be non-degenerate itself. This is \cite[Theorem 2.5]{alvMei21}, which is proved in \cite[Section 8]{alvMei21}. The proof of the main proposition needed in the proof of \cite[Theorem 2.5]{alvMei21}, Proposition 8.2 of the same paper, may also be given by means of Morse theory. In essence, one wants to approximate a degenerate Hamiltonian diffeomorphism $\varphi$ by a non-degenerate one $\varphi'$, which coincides with $\varphi$ on a small neighbourhood of the non-degenerate fixed points $x_1, \dots, x_k$ constituting the given braid. In order to conclude that \cite[Theorem 2.5]{alvMei21} holds it is then necessary to prove that these fixed points are $\varepsilon$-quasi-isolated (a variation of the property in Definition \ref{defn:epsIsol}) for $\varphi'$. Alves and Meiwes do so by showing an isolation property for $\{x_i\}_{i=1, \dots, k}$: any Floer trajectory positively or negatively asymptotic to one of the $x_i$ cannot have energy smaller than $\varepsilon$. In the context of Morse theory an analogous conclusion may be inferred by a simple application of the Morse lemma. Denote $y_1, \dots, y_k$ the critical points of the generating function $h$ of $\varphi'$ which correspond to $x_1, \dots, x_k$. They are in fact non-degenerate, so fix for each of them a Morse neighbourhood $U_i\ni y_i$. Define for each $y_i$ the stable manifold of $y_i$ for $-\nabla h$, $W^s_{-\nabla h}(y_i)$(resp. the unstable manifold of $y_i$ for $-\nabla h$, $W^u_{-\nabla h}(y_i)$). Consider then\begin{equation*}
        U_i^+:=U_i\cap W^s_{-\nabla h}(y_i),\, U_i^-:=U_i\cap W^u_{-\nabla h}(y_i),
    \end{equation*}and let $c_i:=h(y_i)$. By Morse Lemma there exists $\varepsilon_i>0$ such that $ S_i^{\pm\varepsilon_i}:=h^{-1}(c_i\pm \varepsilon_i)\cap U_i^\pm$ is a sphere\footnote{It may technically be empty, if we are in the presence of a relative maximum or minimum. If that is the case, the whole proof works anyway.}, and since such $\varepsilon_i$ are finitely many we may simply take the smallest one, $\varepsilon>0$, satisfying this property for all $y_i$ and $U_i^\pm$. Now, a negative gradient line $\gamma$ which is positively (resp. negatively) asymptotic to $y_i$ has to intersect $S_i^{\varepsilon}$ (resp. $S_i^{-\varepsilon}$): $\gamma$ is not contained in $U_i$ but is fully contained in $W^s_{-\nabla h}(y_i)$ (resp. $W^u_{-\nabla h}(y_i)$) by definition, and to conclude the argument then appeal to the Jordan-Brouwer Separation Theorem, each (un)stable manifold being homeomorphic to an Euclidean space. Since negative gradient trajectories have to intersect such spheres, we obtain the isolation property by the monotonicity of $h$ on trajectories of $-\nabla h$ and the fact that by definition the energy of a negative gradient trajectory is nothing but the difference of the values of $h$ at the positive and negative limit.

    Once the separation property is proved, one can follow the remainder of the proof of Corollary 1.3 as in \cite{alvMei21}.
\end{rem}
\vspace{0.5 cm}

The structure of the paper is as follows: in Section \ref{sec:prerequisites} we report standard facts about braids, twist maps and generating functions quadratic at infinity, and we set up in detail the Morse theory we need (again, this is classical material). In Section \ref{sec:leCGF} we recall the construction of twist generating functions, we give a braid-theoretical interpretation of the points of the vector space on which such generating functions are defined, and we show that twist generating functions may be seen, after some elementary operations, as GFQIs. In Section \ref{sec:definitionFiltration} we recall the properties of a dominated splitting associated to a twist generating function, and we use this to prove Theorem \ref{thm:A}. Theorem \ref{thm:B} is proved in the same section, relying on a Hofer-continuity folklore result for generating functions that we prove. In the Appendix we define a filtration as in Theorem \ref{thm:A} on the Floer complex of a non-degenerate Hamiltonian diffeomorphism of a surface using material from \cite{hwz98}, and we show that the result agrees with our definition using generating functions.

\subsection{Relations to other work}
A fundamental work we would like to mention is the one by Connery-Grigg \cite{conn21}. He used the machinery provided by Hofer-Wysocki-Zehnder \cite{hwz98} and Siefring \cite{sie08},\cite{sie11} to deeply analyse the geometry of the Floer complex. On the one hand he manages to construct singular foliations obtained from the Floer differential (this is connected to Le Calvez's work on transverse foliations), and on the other he finds a topological characterisation of the generators of the Floer complex representing the fundamental class, generalising previous work by Humilière, Le Roux and Seyfaddini \cite{hlrs2016}. Most of the technical lemmata in \cite[Section 3.2.1]{conn21} may be cast in the language of generating functions, and the proofs adapted with ease within the Morse-theoretical context we provide, building on the results by Le Calvez.
\vspace{0.2 cm}

We wish to mention as well that a related linking filtration exists in Embedded Contact Homology and the related theory of PFH. It measures linking number with respect to a fixed Reeb orbit. So far, to the best of the author's knowledge, it has found a few applications: see for instance the paper where Hutchings introduced it \cite{hut16} and other works by Nelson and Weiler \cite{nelWei23}\cite{nelWei24}. In particular, Hutchings's result has been re-proved by Le Calvez using generating functions, see \cite{leC23} and its recent generalisation, due to Bramham and Pirnapasov \cite{braPir24}.
\vspace{0.25 cm}

\paragraph{Acknowledgments}
This text grows out of the authour's doctoral dissertation. He would like to thank his adviser, Vincent Humilière, for his constant encouragement and questions; Alberto Abbondandolo and Vincent Colin for refereeing the dissertation; Marcelo Alves, Matthias Meiwes and Dustin Connery-Grigg for their interest. He also thanks an anonymous referee for highlighting a missing technical point, now explained in Remark \ref{rem:techRem}, and Marcelo Alves again for pointing out an imprecision in the definition of continuation maps.

\section{Technical prerequisites}\label{sec:prerequisites}
\subsection{Braid groups}\label{sec:braids}
Braid groups are classical objects in group theory. Typical references on the subject include the book \cite{kasTur08} and the paper \cite{gon11}. The name of the group was coined by Emil Artin in \cite{art25}, where he gave an abstract definition in terms of generators and relations: for $k\geq 2$,
\begin{equation}\label{eq:braidPres}
\mathcal{B}_k = \left \langle \sigma_1, \dots, \sigma_{k-1} \middle\vert
\begin{aligned}
  |i-j| > 1\Rightarrow&\, \sigma_i\sigma_j  = \sigma_j\sigma_i  \\
  1 \leqslant i\leqslant k-2\Rightarrow&\,\sigma_i\sigma_{i+1} \sigma_i  = \sigma_{i+1} \sigma_i\sigma_{i+1}
\end{aligned}
\right\rangle.
\end{equation}
There are more geometrical ways of seeing the braid groups, the one we are going to use in the present paper being the following one. The permutation group on $k$ letters $\mathfrak{S}_k$ admits a faithful action on the complement in $\D^k$ of the set
\begin{equation*}
    \tilde\Delta=\Set{(x_1, \dots, x_k)\in\D^k\vert \exists i\neq j, x_i=x_j}
\end{equation*}
and we call the quotient
\[
\mathrm{Conf}^k(\D):=(\D^k\setminus\tilde\Delta)/\mathfrak{S}_k
\]
the configuration space of $k$ unordered points of the disc. It turns out that $\pi_1(\text{Conf
}^k(\D))=\mathcal{B}_k$. This means essentially that after choosing $k$ base points $p_1, \dots, p_k$ on $\D$, one can define an element of $\mathcal{B}_k$ uniquely as the homotopy class of a path $\gamma: [0, 1]\rightarrow\D^k\setminus \tilde\Delta$ such that there is $\sigma\in\mathfrak{S}_k$ verifying $\gamma(1)=\sigma \gamma(0)$, and that any braid may be realised this way. The multiplication in $\mathcal{B}_k$ then corresponds to concatenation of braids the obvious way.

There is a quotient homomorphism $\mathcal{B}_k\rightarrow \mathfrak{S}_k$, mapping positive and negative generators to the same transposition, or equivalently identifying $\sigma_i$ with $\sigma_i^{-1}$. Its kernel is called $P_k$, the set of pure braids. One can visualise pure braids as braids such that, following the strands, bring each basepoint on itself.

In the following sections of this work we shall need the definition of linking number\footnote{In the low-dimensional topology literature it may also be called ``exponent''.} \[
\mathrm{lk}: \mathcal{B}_k\rightarrow\Z
\]
which is the only morphism of groups whose value on all the generators $\sigma_i$ is 1. Remark that this convention, very natural from the algebraic viewpoint, requires a different normalisation from the standard one in topology: choosing basepoints $0$ and $\frac{1}{2}$ in $\D$, the braid represented by $t\mapsto [0, \frac{1}{2}\exp\left(2\pi i t\right)]$ has linking number 2 according to our definition, while it is customarily used as example of loop with linking number 1 in geometric settings.

Another possible definition for the linking number, which only applies to pure braids, is the following (see \cite{conn21} for a deeper explanation, including the case of capped braids). Consider a lift to $\D^k$ of a pure braid with $k$ strands $b=[\gamma_1,\dots, \gamma_k]$, be it $\tilde b$, and take any homotopy
\begin{equation*}
    h=(h_1,\dots, h_k):[0, 1]\times[0, 1]\rightarrow \D^k
\end{equation*}
starting at the trivial braid at the basepoints (each strand is constant there) and ending in $\tilde b$. One can prove that
\begin{equation}\label{eq:lkHomotopyIntersections}
    \mathrm{lk}(b)=\sum_{i\neq  j}(h_i\pitchfork h_j)
\end{equation}
where on the right we count intersections with signs.
\begin{rem}
    Each transverse intersection in the count corresponds to a linking difference of 2, due to our normalisation $\mathrm{lk}(\sigma_j)=1$. With the more usual convention one considers the sum indexed on $i< j$ or divides the sum above by 2.
\end{rem}
By homotopy invariance of the intersection product, this quantity does not depend on the choice of $(h_i)$; two different lifts of $b$ being connected by a permutation of the strands, this definition does not depend on the choice of $\tilde b$ either. The linking number then quantifies how far a braid is from being trivial somehow, quantifying the failure of any homotopy $h$ as above to be a homotopy between the trivial braid to $\tilde b$ through braids.

\begin{defn}
    A homotopy $h: [0, 1]\times [0,1]\rightarrow \D^k$ between the lifts of two braids is called a braid cobordism. If its image is contained outside the fat diagonal of $\D^k$, we say that $h$ is a braid isotopy.
\end{defn}

\subsection{Twist Maps and their Generating Functions}\label{sec:twistMaps}
The content of this section is classical: a good introductory reference may be found for instance in \cite{sib04}.

\begin{defn}
\label{defn:twistMap}
    Let $\phi:\R^2\rightarrow \R^2$ be an orientation-preserving diffeomorphism. We say that $\phi$ is a twist map if, denoting $\phi(x_0, y_0)=(x_1, y_1)$, we have the inequality\begin{equation*}
        \frac{\partial x_1}{\partial y_0}>0.
    \end{equation*}
\end{defn}
Geometrically, this may be visualised as follows: given any vertical line of constant $x$ coordinate, it is sent to the graph of a strictly increasing function. The definition may be extended to diffeomorphisms realising the inequality $\frac{\partial x_1}{\partial y_0}<0$, in which case vertical lines are mapped to graphs of strictly decreasing functions.
\begin{rem}
    This interpretation justifies the French denomination of twist maps: they are called \textit{applications déviant la verticale}, literally ``functions deviating the vertical''. The vertical will be deviated \textit{à droite} (``to the right'') or \textit{à gauche} (``to the left'') in respectively the first and second case.
\end{rem}
\begin{es}
    The prototypical example of a twist map is the Dehn twist
    \begin{equation*}
        D: (x, y)\mapsto (x+y, y).
    \end{equation*}
    It is apparent that such a diffeomorphism is a twist map (and in fact it is \textit{déviant la verticale à droite}). The clockwise rotation
    \begin{equation*}
        R:(x, y)\mapsto (y, -x)
    \end{equation*}
    is also a twist map, again \textit{à droite}.
\end{es}
Let us now moreover assume that the twist map $\phi$ is symplectic. In such a case, there exists what is called a generating function for $\phi$: it is a map $h:\R^2\rightarrow \R$ such that
\begin{equation}\label{eq:twistEq}
    \forall (x, y), (x', y')\in\R^2,\,\phi(x, y)=(x', y')\Leftrightarrow\begin{cases}
        y=-\partial_1 h(x, x')\\
        y'=\partial_2 h(x, x')
    \end{cases}.
\end{equation}

\begin{es}
    An example of generating function for the Dehn twist above is\begin{equation}
        \R\times\R\rightarrow \R,\,\,\, (x, x')\mapsto \frac{1}{2}(x'-x)^2
    \end{equation}
    while the clockwise rotation is generated by 
    \begin{equation}
    \R\times\R\rightarrow \R,\,\,\, (x, x')\mapsto -xx' . 
    \end{equation}
\end{es}
\begin{nota}
    A generating function as above will be referred to as ``twist generating function''.
\end{nota}

\subsection{Generating functions quadratic at infinity}\label{subsec:genFuns}
\subsubsection{Basic definitions and properties}
The first construction we would like to introduce is the one of generating functions, used to study Lagrangian intersections in cotangent bundles.

In the following, $M$ is going to be a closed manifold or $\R^{2n}$. Endow $T^*M$ with its canonical exact symplectic structure.

\begin{defn}[Generating functions]\label{defn:GF}
    Let $L\subset T^*M$ be an exact Lagrangian submanifold. We say that $S:E\rightarrow\R$ is a generating function for $L$ if:\begin{itemize}
        \item[\textit{i})] $\pi: E\rightarrow M$ is a real vector bundle on $M$;
        \item[\textit{ii})] $S$ is \textit{vertically transverse} to 0. By this we mean that we define the vertical differential of $S$, let us denote it by \begin{equation*}
            \partial^V S: E\rightarrow (\ker d\pi)^*
        \end{equation*}
        and we ask that $\partial^V S\pitchfork 0$.
        \item[\textit{iii})] $L$ is the image of the following Lagrangian immersion. Let $\Lambda_S:=\partial^V S^{-1}(0)$ (it is a submanifold of $E$ by previous point), and define the Lagrangian immersion by
        \begin{equation*}
            \iota_S: \Lambda_S\rightarrow T^*M,\,\,\, (x, \xi)\mapsto (x, \partial^H S(x, \xi))
        .\end{equation*}
    \end{itemize}
    In the third point, $\partial^H S(x, \xi)\in T^*M$ is defined as
    \begin{equation*}
        (x, v)\mapsto d_{(x, \xi)}S.v', \text{ for any } v'\in T_{(x, \xi)}E, d_{x, \xi}\pi.v'=v.
    \end{equation*}
\end{defn}

By definition, the critical set of a generating function is in bijection with the intersection points between such Lagrangian and the zero section of $T^*M$, so if we want to bound the cardinality of the latter we are naturally lead to do Morse theory on a generating function. A vector bundle is however never a compact manifold, so we need to work within a class of generating functions which satisfy the Palais-Smale compactness condition.
\begin{defn}\label{defn:QI}
    A generating function $S:E\rightarrow \R$ is said to be quadratic at infinity (and we shall write GFQI for it) if there exists a function $Q:E\rightarrow\R$ which when restricted to the fibres is a non degenerate quadratic form, and such that $S-Q$ is supported on a compact set.
\end{defn}

\begin{rem}
    This definition may in fact be relaxed a bit (see \cite{the99}), but we will not need the extended definition.
\end{rem}

Given a GFQI, $S$, we denote by $\sigma(S)$ its signature, defined to be the signature of the quadratic form at infinity (maximal dimension of a subspace on which it is negative definite).

To answer our intersection-counting question via Morse theory, we need an existence statement: as of now, we are not sure for what class of Lagrangian submanifolds one can construct generating functions, if they ever exist at all.

\begin{thm}[Sikorav '87 \cite{sik87}]
    Let us assume that $L\subset T^*M$ admits a generating function: then if $\varphi\in\Ham(T^*M)$, so does $\varphi(L)$.
\end{thm}
\begin{rem}
    It is in fact possible to prove (see \cite{the99}) that the map from generating functions to Lagrangian submanifolds is an infinite dimensional smooth Serre fibration in a precise sense.
\end{rem}
The set of Lagrangians which admit generating functions is therefore closed under Hamiltonian isotopies. Since the zero section $0_M\subset T^*M$ is generated by any non degenerate (fibre-independent) quadratic form, at the very least any Hamiltonian deformation of the zero section admits a generating function.

Existence of a generating function alone however does not guarantee the possibility of counting intersections by means of Morse theory, as mentioned above. Moreover, it is not clear how different generating functions for the same Lagrangian submanifold are related to one another. We introduce then certain elementary operations using which, given a generating function, we may obtain new ones for the same Lagrangian submanifold.

\begin{defn}\label{defn:elOp}
    Assume $S:E\rightarrow \R$ is a generating function on $\pi: E\rightarrow M$ for the Lagrangian $L\subset T^*M$ We call the following operations elementary:\begin{itemize}
        \item[(\textit{Shift})] If $c\in \R$, $S+c$ is a new generating function for $L$, defined again on $E$;
        \item[(\textit{Gauge equivalence})] If $\phi$ is a self-diffeomorphism of $E$ which preserves the fibres of the projection $\pi$, \begin{equation*}
            S\circ \phi: E\rightarrow \R
        \end{equation*} is a new generating function for $L$, defined on the bundle \begin{equation*}
            \pi\circ\phi: E\rightarrow M.
        \end{equation*}
        \item[(\textit{Stabilisation})] If $Q: E'\rightarrow M$ is a non degenerate quadratic form on the vector bundle $\pi':E'\rightarrow M$, then\begin{equation*}
            S+Q:E\oplus E'\rightarrow \R
        \end{equation*}
        is a generating function for $L$, defined on\begin{equation*}
            \pi\oplus \pi': E\oplus E'\rightarrow M.
        \end{equation*}
    \end{itemize}
\end{defn}
\begin{rem}
    We may, without loss of generality, always assume to be adding trivial bundles for the (Stabilisation) operation. In fact, if $E'\rightarrow M$ is a vector bundle, there exists another vector bundle on $M$, say $E''$, such that $E'\oplus E''$ is trivial. Moreover, it is remarked in \cite{the99} that any non degenerate quadratic form is equivalent diffeomorphic by a gauge equivalence to one which does not depend on the point of the base, i.e. if the vector bundle $E$ is globally trivial, and $p, q\in B$, then $Q\vert_{\pi^{-1}(p)}=Q\vert_{\pi^{-1}(q)}$.
\end{rem}

We may now state Viterbo's Uniqueness Theorem:
\begin{thm}[Viterbo, \cite{viterbo92}]\label{thm:vitUniqueness}
A Lagrangian submanifold $L\subset T^*M$ is said to have the GFQI-uniqueness property if, given $S: E\rightarrow \R$, $S':E'\rightarrow \R$ generating functions for $L$, there exist two finite sequences of elementary operations taking both $S$ and $S'$ to the same $S'': M\times \R^N\rightarrow \R$, GFQI for $L$. Then:
\begin{itemize}
    \item[\textit{i})] If $L$ has the GFQI-uniqueness property, and $\varphi\in \Ham(T^*M)$, then so does $\varphi(L)$;
    \item[\textit{ii})] The zero section $0_M$ has the GFQI-uniqueness property.
\end{itemize}
\end{thm}

\begin{proof}
    For a complete proof of the Theorem, see \cite{the99}.
\end{proof}

Generating functions may be used to describe compactly supported Hamiltonian diffeomorphisms of $(\R^{2n}, \omega=\sum_i dx^i\wedge dy^i)$: denoting by $\overline {\R^{2n}}$ the symplectic manifold $(\R^{2n}, -\omega)$, if $\Delta$ is the diagonal of $\overline {\R^{2n}}\oplus \R^{2n}$, there is a symplectic identification
\begin{equation*}
   \overline {\R^{2n}}\oplus \R^{2n} \rightarrow T^*\Delta
\end{equation*}
mapping $\Delta$ to the zero section, given by
\begin{equation}\label{eq:identCotDelta}
    (x, y, X, Y)\mapsto \left(x, Y, y-Y, X-x\right).
\end{equation}
We have endowed above $T^*\Delta$ with its tautological exact symplectic form. Given a Hamiltonian diffeomorphism of $\R^{2n}$ with compact support, we push forward its graph to $T^*\Delta$ by the symplectic identification above. One may then find a generating function for the Lagrangian deformation of the zero section in $T^*\Delta$ one obtains this way. Fixed points of the Hamiltonian diffeomorphism are in bijection with the Lagrangian intersections between the zero section and its deformation, and thus with the critical points of the generating function.
Let $\varphi\in \Ham_c(\R^{2n})$, and\begin{equation*}
    h: \Delta\times \R^{N}_\xi\rightarrow \R
\end{equation*}
a generating function of its graph in $T^*\Delta$. The map in (\ref{eq:identCotDelta}) yields then the equivalence
\begin{equation}\label{eq:discrHamEqn}
    \varphi(x, y)=(X, Y)\Leftrightarrow\begin{cases}
        \partial_{\xi}h(x, Y; \xi)=0\\
        X-x=\partial_Yh(x, Y; \xi)\\
        Y-y=-\partial_xh(x, Y; \xi)
    \end{cases}.
\end{equation}
\begin{rem}
    Let us remark that in principle the kind of generating functions we have just talked about is a different one from the ones earlier in the previous Section, for twist maps.
\end{rem}
\begin{rem}
    A diffeomorphism $\varphi\in \Ham_c(\R^{2n})$ is non degenerate if and only if for any GFQI $S$ with quadratic form $Q$ that represents it, $S-Q$ is Morse in the interior of its support.
\end{rem}
We now quote a Lemma as reported by Brunella about composition formulae for GFQI.
\begin{lem}[\cite{bru91}]\label{lem:bru}
    Let $L\subset T^*\R^{2n}$ be an immersed Lagrangian submanifold with a GFQI $S: \R^{2n}\times \R^k\rightarrow \R$ coinciding with $Q: \R^k\rightarrow \R$ at infinity. Let $\varphi$ be a compactly supported Hamiltonian diffeomorphism of $\R^{4n}$ with generating function $F: \R^{2n}\times \R^{2n}\rightarrow \R$. Then $\varphi(L)$ has a GFQI with fibres of dimension $4n+k$.
\end{lem}
This Lemma is proved essentially cutting a Hamiltonian isotopy between the identity and $\varphi$ into $\mathscr{C}^1$-small pieces. Each of these portions admits a GFQI, and one can glue them.

\subsubsection{Morse theory for GFQIs on $\R^{2n}$}\label{subsec:morseGFQI} Trying to define Morse theory for a GFQI $S: \R^{2n}\times \R^k\rightarrow \R$ of a non degenerate, compactly supported Hamiltonian diffeomorphism $\varphi\in \Ham_c(\R^{2n})$ , one immediately runs into the problem that $S$ is never going to be Morse on the whole of $\R^{2n}$. It is possible to try to define the whole Morse complex as a limit of complexes of perturbations in a given family. We are not going to proceed this way, as considering the limit one has to take the homology first for the continuation maps to be well defined,  thus losing the correspondence generator$\leftrightarrow$fixed point.

An alternative approach is to fix a single perturbation of $S$, assuming it has only finitely many critical points: this is possible for instance making the hypothesis that the perturbation is radial with respect to the $\R^{2n}$-coordinate at infinity, without critical points outside a compact set (this can be seen as an application of the standard genericity result for Morse functions in the compact case, compactifying  $\R^{2n}$ to $\sphere^{2n}$). It is possible to read this paper supposing that all functions be really Morse, and the pairs function-metric be both Morse-Smale and Palais-Smale.

We now present an approach that has the advantage of retaining perturbation-independent information only. It has been inspired by an analogous constrtuction in \cite{alvMei21}.

\begin{defn}
    Let $\varphi\in \Ham_c(\R^{2n})$, and $x\in \Fix(\varphi)$. For a fixed generating function $S$ of $\varphi$, the $S$-action of $x$ is defined to be the critical value of the critical point of $S$ corresponding to $x$. The set of all critical values of $S$ is called the spectrum of $S$, and will be denoted by $\mathrm{Spec}(S)$. The set $\mathrm{Spec}(\varphi)$ is defined to be equal to $\mathrm{Spec}(S)$ for any generating function $S$ whose degenerate critical points (corresponding to fixed points outside the support of $\varphi$) have critical value 0.
\end{defn}
\begin{rem}
    For a fixed $\varphi\in\Ham_c(\R^{2n})$, the spectrum depends on the generating function we choose up to translation, because of the (Shift) operation. If we restrict our attention to GFQIs however the spectrum is well defined.
\end{rem}

\begin{defn}
    A compactly supported $\varphi\in\Ham_c(\R^{2n})$ is said to be non degenerate if it is non degenerate on the interior of its support.
\end{defn}

Assume that $\varphi$ is non degenerate, and that each non degenerate fixed point has $S$-action different from 0. This condition is generic in $\Ham_c(\R^{2n})$. Perturb the generating function $S$ as above (we call the perturbation $S$ again). Fix any $a<b$ with $0\not\in (a, b)$, or equivalently $ab>0$, allowing $a=-\infty$ or $b=+\infty$ when possible, and any $g$ making $(S, g)$ both Palais-Smale and Morse-Smale. One can then define
\begin{equation*}
    CM^a(S, g; \Z):=\bigoplus_{x\in\mathrm{Crit}(S), S(x)<a}\Z\cdot x
\end{equation*}
and similarly $CM^b(S, g; \Z)$.

\begin{rem}
    Assume the perturbation of $S$ to be small enough. If $a>0$ both subcomplexes contain the non degenerate critical points obtained from perturbing the degenerate critical points of 0 $S$-action. If $b<0$, such points do not appear in either subcomplex.
\end{rem}
We then define the Morse complex in the action window $(a, b)$ to be
\begin{equation}\label{eq:CMSab}
    CM^{(a, b)}(S, g;\Z):=CM^b(S, g; \Z)/CM^a(S, g; \Z)
\end{equation}
where the differential counts gradient lines connecting generators. In the text, whenever we write
\begin{equation*}
    CM(S, g; \Z)
\end{equation*}
for a generating function $S$ of a given compactly supported Hamiltonian diffeomorphism, we mean the collection
\begin{equation}\label{eq:CMS}
    \left(CM^{(a, b)}(S, g; \Z)\right)_{a<b, 0\not\in(a, b)}
\end{equation}
as defined in (\ref{eq:CMSab}).

We now shift the grading of the complex much like in the work of Traynor \cite{tra94} (this operation does not affect what has been done above). Let $\sigma(S)$ the signature of a GFQI $S$. If $x\in \mathrm{Crit}(S)$ has Morse index $k$, it is a generator of $CM(S, g; \Z)$ in degree 
$k-\sigma(S)$. With this convention, the three elementary operations of Definition \ref{defn:elOp} induce isomorphisms of graded complexes. The grading shift is clearly needed due to the stabilisation operation, which in general changes the Morse index.

\paragraph{Continuation maps}We now sketch the construction for Morse continuation maps.
Let $S_0, S_1$ be GFQI $\R^{2n}\times \R^k\rightarrow \R$ for two compactly supported non degenerate Hamiltonian diffeomorphisms. Let $Q_0$ and $Q_1$ be the two asymptotic non degenerate forms, which we may assume to be independent of the base point $x\in\R^{2n}$. We require that the $S_i$ be defined on the same space (we did so above implicitly), and that $Q_0=Q_1$: up to stabilisation they have the same signature, and then apply a fibre-preserving diffeomorphism taking an orthogonal basis of one to an orthogonal basis of the other. As highlighted above, $S_0$ and $S_1$ will be non degenerate only in the interior of the supports of $S_i-Q_i$. We ignore this problem in the exposition, using the above method. Fix Riemannian metrics $g_i$ such that, up to small perturbations, the $(S_i, g_i)$ are Morse-Smale and Palais-Smale pairs.

Define the function \begin{equation*}
    S: [0, 1]\times\R^{2n}\times \R^k\rightarrow \R,\qquad S(s, x):=(1-s)S_0(x)+sS_1(x)
\end{equation*}
and the metric on $[0, 1]\times\R^{2n}\times \R^k$
\begin{equation*}
    g(s, x):=(1-s)g_0(x)+sg_1(x)+ds^2
\end{equation*}
Perturb $S$ so that it has a smooth extension to $(-\delta, 1+\delta)\times \R^{2n+k}$, for a small $\delta>0$, satisfying $\partial_sS(s, x)\vert_{s=0}=\partial_sS(s, x)\vert_{s=1}=0$. Likewise, extend $g$ on the same set. Define a smooth $\rho:(-\delta, 1+\delta)\rightarrow \R$ such that for a $\varepsilon>0$ we have
\begin{itemize}
    \item $\rho(0)=\norm{S_0-S_1}_{\infty}+\varepsilon$;
    \item $\rho(1)=0$;
    \item $\rho'(s)<0$ for all $t\in(0, 1)$;
    \item $0$ is a point of global maximum for $\rho$, and $1$ one of global minimum.
\end{itemize}
Remark that such function $\rho$ may be constructed for an arbitrarily small value of $\varepsilon$.
Then we also may assume that
\begin{equation}\label{eq:ineqh}
   \forall s\in (-\delta, 1+\delta), \qquad \abs{\rho'(s)}>\norm{S_1-S_0}_\infty.
\end{equation}
Because of this inequality, studying the Morse complex of
\begin{equation*}
    (s, x)\mapsto \rho(s)+S(s, x)
\end{equation*}
together with the perturbation of the metric $g$ above, one may define a map
\begin{equation*}
    CM(S_0, g_0; \Z)\rightarrow CM(S_1, g_1; \Z).
\end{equation*}
In fact, (\ref{eq:ineqh}) forces negative gradient lines to connect critical points of $S$ with $s=0$, which are in bijection with critical points of $S_0$, to critical points of $S$ with $s=1$, which are secretly critical points of $S_1$.

The complexes above are defined if the $S_i$ are Morse. Otherwise, we are forced to look at the homology of the complexes $CM^{(a, b)}(S_i, g_i; \Z)$. To describe continuation maps in this case, we need to be able to describe how the value of the filtration changes.\begin{lem}
    Let $x_0\in \mathrm{Crit}(S_0)$ be a generator of the Morse complex for $(S_0, g_0)$. Assume $x_1\in \mathrm{Crit}(S_1)$ is a generator for the target Morse complex such that there exists a negative gradient line for $(\rho+S,g)$
    \begin{equation*}
        \gamma:\R\rightarrow (-\delta, 1+\delta)\times\R^{2n+k}, \,\,\, \dot \gamma(\sigma)=-\nabla^{g}(\rho+S)(\gamma(\sigma))
    \end{equation*}
    negatively asymptotic to $(0, x_0)$ and positively asymptotic to $(1, x_1)$
    \begin{equation*}
        \lim_{\sigma\to-\infty}\gamma(\sigma)=(0, x_0),\,\, \lim_{\sigma\to+\infty}\gamma(\sigma)=(1, x_1).
    \end{equation*}
    Then \begin{equation}\label{eq:SHofLipschi}
        S_1(x_1)-S_0(x_0)\leq \norm{S_0-S_1}_\infty+\varepsilon
    \end{equation}
    where $\varepsilon$ appears in the definition of $\rho$.
\end{lem}
\begin{proof}
    Because $\gamma$ is a negative gradient line for $\rho+S$, clearly
    \begin{equation*}
        (\rho+S)(1, x_1)=(\rho+S)(\gamma(+\infty))\leq (\rho+S)(\gamma(-\infty))=(\rho+S)(0, x_0).
    \end{equation*}
    By definition therefore
    \begin{equation*}
        S_1(x_1)-S_0(x_0)\leq \rho(0)-\rho(1) = \norm{S_0-S_1}_\infty+\varepsilon.
    \end{equation*}
\end{proof}

One inconvenient of the definition in (\ref{eq:CMS}) is that continuation maps get more complicated to describe, since the chain complex is not defined for all action values at once. The Lemma above lets us define the continuation maps for fixed intervals of values of the Morse functions.
\begin{cor}\label{cor:morContActWin}
    Let $a<b$ with $ab>0$, and write
    \begin{equation*}
        \lambda:=\norm{S_0-S_1}_\infty.
    \end{equation*}
    Then, for $\varepsilon>0$ small enough, the following continuation maps are defined if (and only if) both $ab>0, (a+\lambda)(b+\lambda)>0$:
    \begin{equation*}
        CM^{(a, b)}(S_0, g_0; \Z)\rightarrow CM^{(a+\lambda+\varepsilon, b+\lambda+\varepsilon)}(S_1, g_1; \Z).
    \end{equation*}
\end{cor}
\begin{proof}
    The action shift is clear by the above Lemma. We only need to be mindful about the definition of the Morse complex of $\rho+S$: we have to show that as soon as the Morse complexes for the $S_i$ above are defined, then so is the one of $\rho+S$. To define the morphism we use the differential of the Morse complex
    \begin{equation*}
        CM^{(a+\lambda+\varepsilon, b+\lambda+\varepsilon)}(\rho+S, g; \Z).
    \end{equation*}
   For $\varepsilon$ small enough there are no degenerate critical points of $\rho+S$ in this interval. Critical points in $\{0\}\times \R^{2n+k}$ coincide with critical values of $S_0$ by construction, but the critical value is increased by $\lambda$, so the morphism is indeed defined where claimed.
\end{proof}
In particular this Corollary shows that continuation maps in this framework may therefore be defined for small $\mathscr{C}^0$-deformations of the generating function only, the size of which depends on the actual action window $(a,b)$ we start from.

\section{Setup: twist generating functions and linking}\label{sec:leCGF}
In this section we show how generating functions obtained by Le Calvez in his works \cite{leCalvez91}, \cite{leCalvez1999} fit our Morse-theoretical picture. Along the way, we describe the Lyapunov property of the linking number with respect to generating functions of Le Calvez type.
\subsection{Construction of the generating function}

Let $\varphi\in\Ham_c(\D)$. As Patrice Le Calvez points out in \cite{leCalvez91}, it is possible to describe $\varphi$ as a composition of Hamiltonian twist maps of the plane, which therefore have generating functions in the sense of Section \ref{sec:twistMaps}. A way of doing so is for instance choosing a Hamiltonian isotopy $(\varphi_t)_{t\in[0, 1]}$ ending at $\varphi$, and cutting it into pieces $\varphi_0, \dots, \varphi_{n-1}$ so that each $\varphi_i$ is $\mathscr{C}^1$-close to the identity and
\begin{equation*}
    \varphi=\varphi_{n-1}\circ\cdots\circ\varphi_0.
\end{equation*}
If the $\varphi_i$ are close enough to the identity and
\begin{equation}R:\R^2\rightarrow\R^2, (x, y)\mapsto (y, -x)
\end{equation}
is the positive rotation, then $\varphi_i\circ R^{-1}$ is a twist map (Definition \ref{defn:twistMap}) for all $i$. We thus obtain a decomposition of $\varphi$ as product of twist maps
\begin{equation}\label{eq:decompTwist}
    \varphi=(\varphi_{n-1}\circ R^{-1})\circ R\circ (\varphi_{n-2}\circ R^{-1})\circ R\circ \cdots \circ (\varphi_0\circ R^{-1})\circ R.
\end{equation}
\begin{rem}
    Le Calvez in \cite{leCalvez1999} uses Dehn twist instead of rotations. This is not going to affect the properties we exploit here in any way. What is important is the twist condition, that here still holds.
\end{rem}
Write $\Phi_{2i}:=R$ and $\Phi_{2i+1}:=\varphi_{i}\circ R^{-1}$ in order to simplify the notation:
\begin{equation}
    \varphi=\Phi_{2n-1}\circ \cdots \circ \Phi_{0}.
\end{equation}

Now, we let $h_i:\R^2\rightarrow \R$ be the generating function for $\Phi_i$. Recall that by definition then the following equations are verified:
\begin{equation}
    \forall (x, y), (x', y')\in\R^2,\,\Phi_i(x, y)=(x', y')\Leftrightarrow\begin{cases}
        y=-\partial_1 h_i(x, x')\\
        y'=\partial_2 h_i(x, x')
    \end{cases}.
\end{equation}
The critical points of $h_i$, as expected, are in a bijection with the fixed points of $\Phi_i$. Following Le Calvez's notation, we are going to write
\begin{equation*}
    \begin{cases}
        g(x, x')=-\partial_xh(x, x')\\
        g'(x, x')=\partial_{x'}h(x, x')
    \end{cases}.
\end{equation*}

One can then take the sum of the $h_i$ to find a function whose critical points (equivalently, fixed points of its gradient flow) are in bijection with the fixed points of $\varphi$. We define a function
\begin{equation}
    h: E:=\R^{2n}\rightarrow \R,\,\, h(x)=\sum_{i=0}^{2n-1}h_i(x_i, x_{i+1})
\end{equation}
where we set $x_{2n}:=x_0$.
\begin{nota}
    We shall call a function $h$ of this kind a ``twist generating function'' as well. A priori however it does not satisfy (\ref{eq:twistEq}).
\end{nota}

We endow $E$ with the standard Euclidean scalar product, let us call it $g$; let $\xi$ be the negative gradient vector field defined by the pair $(h, g)$\begin{equation}
\xi=-\nabla_gh.
\end{equation}

\begin{defn}
    We say that $h$ obtained as above is a twist generating function.
\end{defn}

The tool we use to calculate the linking number between fixed points is a function defined on an open subset of $\R^{2n}$. If
\begin{equation}
    V=\Set{x\in \R^{2n}\vert x_i\neq 0\,\,\forall i\in \Z}
\end{equation}
we define
\begin{equation}
    L(x)=\frac{1}{4}\sum_{i=0}^{2n-1}(-1)^i\mathrm{sgn}(x_i)\mathrm{sgn}(x_{i+1})\in\Z.
\end{equation}
This function admits a continuous extension to the subset
\begin{equation}
    W=\Set{x\in \R^{2n}\vert \forall i\in \Z,\,\, x_i=0\Rightarrow x_{i-1}x_{i+1}>0}.
\end{equation}
A first result is the following Lemma:
\begin{lem}
    Let $x^0$ and $x^1$ be critical points of $h$ (equivalently, $\xi(x^i)=0$). Then,\begin{equation}
        L(x^0-x^1)=\frac{1}{2}\mathrm{lk}({x^1},{x^0}).
    \end{equation}
On the right hand-side, ${x^i}\in \Fix(\varphi)$ is the fixed point (or indifferently, 1-periodic orbit) represented by the critical points of $h$ with the same name.
\end{lem}
\begin{rem}
    The factor of one half is absent in \cite{leCalvez1999}. We introduce it here because we choose a different normalisation for the linking number.To see it for instance remark that the linking number of the braid
    \begin{equation*}
        t\mapsto \left[0, \frac{1}{2}\exp(2\pi i t) \right]\in \mathrm{Conf}^2\D
    \end{equation*}
    is 1 according to Le Calvez's conventions, but with our normalisation it is 2 since it corresponds to the square of the generator of $\mathcal{B}_2$.
\end{rem}
\begin{rem}
A consequence of this lemma and the fact that $L$ takes values in the set $\Set{-\floor*{\frac{n}{2}}, \dots, \floor*{\frac{n}{2}}}$ is that the length of the decomposition of $\phi$ necessarily depends on the maximal (in absolute value) linking number of two different fixed points. 
\end{rem}

The first main result is the following theorem, which may be interpreted as a finite-dimensional version of the positivity of intersections for holomorphic curves in 4-manifolds (see \cite[Appendix E]{mcDSal16})
\begin{thm}[Le Calvez, \cite{leCalvez91}]\label{thm:LeCalLyapunov}
    Let $x^i\in E$, $i=0, 1$, and denote by $x^i(t)$ the images of $x^i$ under the flow of $\xi$. Then the function\begin{equation}
        t\mapsto L(x^0(t)-x^1(t))
    \end{equation}
    is defined and continuous outside a finite set of points, and if it is not defined for a $t\in \R$ then
    \begin{equation}
        \lim_{\varepsilon\to 0^+}L(x^0(t-\varepsilon)-x^1(t-\varepsilon))<\lim_{\varepsilon\to 0^+}L(x^0(t+\varepsilon)-x^1(t+\varepsilon)).
    \end{equation}
\end{thm}
This theorem is the first step to define the filtration. The goal is, for any $S$ generating function for $\varphi$ and $g$ is a Riemannian metric on the associated vector bundle, to define a function
\begin{equation}
    I:CM(S, g; \Z)\otimes CM(S, g; \Z)\rightarrow \Z
\end{equation}
which is increasing along the differential. The Theorem tells us that, if we take $h$ as generating function (it is yet to be seen in which sense it is one) whenever $y\not\in\partial x$, we have the inequality
\begin{equation}
    I(x \otimes y)\leq I(\partial x \otimes y).
\end{equation}
What remains to define a filtration is removing the assumption that $y\not\in\partial x$, or equivalently defining a notion of self-linking number for a fixed point of $\varphi$ which is consistent with the linking numbers of all other pairs of fixed points in the Morse complex. To carry out this operation, we need to understand the behaviour of $L$ under the linearised dynamics of $\xi$. Before carrying this out, we would like to explain how this function $h$ is in fact a generating function in the classical (à la Viterbo) sense, to fix the ideas. This fact is relevant as it will be exploited in constructing the filtration for any (reasonable) generating function in the following.

\subsection{Collections of points in $E$ as braids}\label{sec:pointsEBraidsLeC}
Given a generating function, there exists an evident function mapping collections of critical points to braids, as we saw in the introduction. What is peculiar about the function $h$ defined on $E$ is that we can associate to (almost) any collection of $k$ distinct points of $E$ a braid type. We use Le Calvez's description from \cite{leCalvez1999}. We define a map
\begin{equation*}
    \gamma: E\times [0, 2n]\rightarrow \R^2
    \end{equation*}
as follows: for $i\in\{0, \dots, n\}$,
\begin{equation*}
    \gamma(x, t)=
    \begin{cases}
    (2i+1-t)(x_{2i}, x_{2i+1})+(t-2i)(x_{2i+2}, x_{2i+1})& t\in[2i, 2i+1]\\
    (2i+2-t)(x_{2i+2}, x_{2i+1})+(t-2i-1)(x_{2i+2},x_{2i+3})& t\in[2i+1, 2i+2]  
    \end{cases}.
\end{equation*}
Here we extend the coordinates by periodicity, i.e. $x_{2n}=x_0, \dots, x_{2n+3}=x_3$.

The function $\gamma$ associates to each point in $E$ a loop in $\R^2$ of length $2n$; this association is furthermore injective. We may reparameterise $\gamma$ so that the associated loop has length 1, and we do so. If we assume now that $x\in W$ we obtain the following crucial relation between $L$ and $\gamma(x, \cdot)$:
\begin{equation}
    L(x)=\frac{1}{2}\mathrm{lk}(0, \gamma(x, \cdot))
\end{equation}
where $0$  denotes the constant loop based at 0 in the plane. Moreover, this identity behaves well when taking differences of points in $E$. In fact, as Le Calvez shows,
\begin{equation}
    L(x-y)=\frac{1}{2}\mathrm{lk}(\gamma(x, \cdot), \gamma(y, \cdot)).
\end{equation}

The pairs of points $x, y\in E$ with $x-y\in W$ are in fact precisely those for which $[\gamma(x, \cdot), \gamma(y, \cdot)]$ is a braid. Alternatively, $x-y\not\in W$ if and only if there exists a time $s\in[0, 1]$ such that
\begin{equation*}
    \gamma(x, t)= \gamma(y, t).
\end{equation*}
Let now $x^1, \dots,  x^k\in E$ be points so that for all $i\neq j\in\{1, \dots, k\}$ one has
\begin{equation*}
    x_i-x_j\in W.
\end{equation*}
The map
\begin{equation*}
    b': (x^1, \dots,  x^k)\mapsto [\gamma(x^1, \cdot), \dots, \gamma(x^k, \cdot)]
\end{equation*}
gives us the braid description we promised.

\begin{lem}\label{lem:braidTypesCoincide}
    Let $x^1, \dots, x^k$ be $k$ distinct critical points of $h$. Then, if $n$ is large enough, we have $b(x^0, \dots, x^k)=b'(x^0, \dots, x^k)$.
\end{lem}
\begin{proof}
    Let $H$ be a Hamiltonian generating $\varphi$ and suppose that $h$ is obtained by cutting the Hamiltonian isotopy defined by $H$ into $\mathscr{C}^1$-small pieces. For every integer $i\in\{0, \dots, n\}$, the curve $\gamma\vert_{[2i, 2i+2]}$ is by construction the union of two short segments. Their length is dependent on the $\mathscr{C}^0$-norm of $\varphi_i$. Denoting by $\phi_i(x,y)_x$ the $x$-coordinate of the point $\varphi_i(x, y)$ for any point $(x, y)$ in the plane, because of the shape of the decomposition in (\ref{eq:decompTwist}), we find that $\gamma\vert_{[2i, 2i+1]}$ is a horizontal segment between
    \begin{equation*}
        \varphi_{2i-1}\circ \cdots \varphi_{0}(x, y)
    \end{equation*}
     and
     \begin{equation*}
         ((\varphi_{2i}\circ\varphi_{2i-1}\circ \cdots \varphi_{0}(x, y))_x, (\varphi_{2i-1}\circ \cdots \varphi_{0}(x, y))_y).
     \end{equation*}
     In the case $i=0$, we have a short horizontal segment connecting
     \begin{equation*}
         (x, y)\text{ to } (\varphi_0(x, y)_x, y).
     \end{equation*}
     Instead, when considering $\gamma\vert_{[2i+1, 2i+2]}$, we see a short vertical segment connecting
     \begin{equation*}
         ((\varphi_{2i}\circ\varphi_{2i-1}\circ \cdots \varphi_{0}(x, y))_x, (\varphi_{2i-1}\circ \cdots \varphi_{0}(x, y))_y)
     \end{equation*}
     to
     \begin{equation*}
         \varphi_{2i}\circ\varphi_{2i-1}\circ \cdots \varphi_{0}(x, y).
     \end{equation*}
     The curve $\gamma$ (after reparameterisation so that it is defined on times in $[0,1]$) is therefore a $\mathscr{C}^0$-approximation of the path
     \begin{equation*}
         t\mapsto \phi^t_H(x, y).
     \end{equation*}
     Applying this remark to each of the $x^j$, for $j=1, \dots, k$, we obtain that the braid
     \begin{equation*}
         t\mapsto [\gamma(x^1, \cdot), \dots, \gamma(x^k, \cdot)]
     \end{equation*}
     is a strand-by-strand $\mathscr{C}^0$-approximation of the braid
     \begin{equation*}
         t\mapsto [\phi^t_H(x^1, y^1), \dots, \phi^t_H(x^k, y^k)]
     \end{equation*}
     where $(x^j, y^j)$ is the point in $\R^2$ associated to the point $x^j$ by definition\footnote{The two $x^j$ are different, the former is a real number, the latter a vector}. The two braids are therefore homotopic as soon as $n$ is big enough.
\end{proof}

\begin{nota}
    By virtue of the Lemma above, we shall now suppress the notation $b'$.
\end{nota}

In the case of homotopies of Morse data resulting in continuation maps between twist generating functions, we may provide an interpretation to the gradient lines as braid cobordisms. We can in fact define a function $\tilde\gamma$ on $E\times [0, 1]\times (-\delta, 1+\delta)$ simply by
\begin{equation*}
    \tilde\gamma(x, t, s):=\gamma(x, t).
\end{equation*}
Repeating the construction above with $\tilde\gamma$ yields the fact that a collection of gradient lines involved in the definition of a continuation map gives (as soon as the positive endpoints are all pairwise distinct) gives a braid cobordism between braids associated to collections of critical points. Remark that the times $s$ at which $L$ is not defined at a difference between two gradient lines correspond precisely to the times at which the braid cobordism fails to be a braid isotopy.

\subsection{Le Calvez's $h$ is a GFQI}
We now show that this $h$ is a generating function for $\varphi$ in the sense of Definitions \ref{defn:GF} and \ref{defn:QI} as seen through Equation (\ref{eq:discrHamEqn}). This is the content of Corollary \ref{cor:hIsGFQI}, stated below.
\vspace{0.5 cm}

A point $(x, y)\in\R^2$ is mapped to $(X, Y)$ if and only if there are (unique) points $(x_i, y_i)\in \R^2, \, i\in \{0, \dots, 2n\}$ such that $(x_0, y_0)=(x, y)$, $(x_{2n}, y_{2n})=(x', y')$, $\varphi_i(x_{i_1}, y_{i-1})=(x_i, y_i)$. This last condition is verified if and only if, since $h_i$ generates $f_i$,\[
\begin{cases}
y_{i-1}=-\partial_{x_{i-1}}h_i(x_{i-1}, x_i)=g_i(x_{i-1}, x_i)\\
y_{i}=\partial_{x_{i}}h_i(x_{i-1}, x_i)=g'_i(x_{i-1}, x_i)
\end{cases}.
\]

Let us assume that the decomposition ends in $R^{-1}\circ R$ (we are not losing in generality, we are essentially adding a trivial factor in the decomposition of $\varphi$ into $\mathscr{C}^1$-small factors); the order is important , since we want our maps in the decomposition to twist positively and negatively alternatively. The decomposition of $\varphi$ more concretely will have to look like this:
\begin{equation*}
   R^{-1}\circ R\circ \varphi_{n-1}\circ R^{-1}\circ \cdots \varphi_0\circ R^{-1}\circ R
\end{equation*}
where each $\varphi_i$ is $\mathscr C^1$-small.

Remember now that $R$ is generated by $(x, x')\mapsto -xx'$ and $R^{-1}$ by $(x, x')\mapsto xx'$. A Le Calvez generating function associated to this decomposition is therefore:
\begin{equation*}
h:\R^{2n+2}\rightarrow \R,\,\,(x_0, \dots, x_{2n+1})\mapsto -x_0x_1+\sum_{i=1}^{2n-1}h_i(x_i, x_{i+1})-x_{2n}x_{2n+1}+x_{2n+1}x_0.
\end{equation*}
We isolated the first and the last two terms because they are the ones we are going to differentiate in the following.

We now define the projection:
\[
\R^{2n+2}\rightarrow \R^2, \, (x_0, \dots, x_{2n+1})\mapsto (x_0, x_{2n+1}).
\]
Requiring the vertical differential to be 0 is then equivalent to asking for $\xi_i=0$ for $i=1, \dots, 2n$, i.e. (compare \cite{leCalvez1999}) for\[
\Phi_{i-1}(x_{i-1}, y_{i-1})=(x_i, y_{i}), \text{ for }i=1,\dots, 2n.
\]
A simple computation then gives:\begin{equation}\label{eq:genFunctionToConclude}
    \begin{cases}
        x_{2n+1}-x_1=\partial_{x_0}h(x_0, \dots, x_{2n+1})\\ -x_{2n}+x_0=\partial_{x_{2n+1}}h(x_0,\dots, x_{2n+1} )
    \end{cases}.
\end{equation}

Let us show how this is enough to conclude. The vanishing of the vertical differential ensures that $x_{2n+1}=y_{2n}$: in fact
\begin{equation*}
    \partial_{x_{2n}}h=0\Leftrightarrow x_{2n+1}=g'_{2n-1}(x_{2n-1}, x_{2n})
\end{equation*}
but, on the other side, we know ($\xi_{2n}=0$) that\begin{equation*}
    \Phi_{2n-1}(x_{2n-1}, y_{2n-1})=(x_{2n}, y_{2n})
\end{equation*}
that which is verified if and only if
\begin{equation*}
    \begin{cases}
        y_{2n-1}=g_{2n-1}(x_{2n_1}, x_{2n})\\
    y_{2n}=g_{2n-1}'(x_{2n_1}, x_{2n})
    \end{cases}
\end{equation*}
and the second equation allows us to conclude that\begin{equation*}
    x_{2n+1}=y_{2n}
\end{equation*}
as desired.

Going back to the formalism of the introduction,write
\begin{equation*}
    (X, Y):=\varphi(x, y)
\end{equation*}
then $(x_{2n}, y_{2n})=(X, Y)$. The equations (\ref{eq:genFunctionToConclude}) above then become\begin{equation*}
   \varphi(x_0, y_0)=(X,Y)\Leftrightarrow \begin{cases}
\partial^Vh=0\\
X-x=\partial_{Y}h(x_0=x,x_1, \dots, x_{2n}, x_{2n+1}=Y)\\
Y-y=-\partial_{x}h(x_0=x,x_1, \dots, x_{2n}, x_{2n+1}=Y)\\

\end{cases} 
\end{equation*}
and this is precisely the same equivalence as in (\ref{eq:discrHamEqn}). This shows that $h$ is indeed a generating function for $\varphi$ in the classical sense, but we do not know yet if we can do Morse theory with it: we need the Palais-Smale condition. A first step is proving that $h$ can be made quadratic at infinity by a linear map of $\R^{2n}$ that is also a gauge equivalence (Definition \ref{defn:elOp}).
\vspace{0.2cm}

If $\varphi$ has compact support, it is easy to see that every term in the decomposition can be assumed to coincide with the positive or the negative rotation outside a compact set (see \cite{leCalvez91}). Write
$\varphi=\Phi_{2n+1}\circ \cdots \circ \Phi_{0}$ as above. The generating function associated to the decomposition, outside of a compact set of $\R^{2n+2}$, coincides with
\begin{equation}\label{eq:hInfty}
    h_\infty(x_0, \dots, x_{2n+1})=\sum_{i=0}^{2n+1}(-1)^{i+1}x_ix_{i+1}
\end{equation}
which is a genuine quadratic form whose kernel has dimension 2: the Gram matrix associated to $h$ is a so-called Jacobi matrix
\begin{equation*}
    \begin{pmatrix}
        0&-1&0&\cdots&0&1\\
        -1&0&1&&&0\\
        0&1&0&\ddots&&\vdots\\
    \vdots&&\ddots&\ddots&-1&0\\
    0&&&-1&0&1\\
    1&0&\cdots&0&1&0
    \end{pmatrix}.
\end{equation*}
Such a bilinear form clearly has a two dimensional kernel, spanned by ($(e_j)$ is here the canonical basis of $\R^{2n+2}$)
\begin{equation*}
    \sum_{i=0}^{n}e_{2i}, \sum_{i=0}^{n}e_{2i+1}.
\end{equation*}
Completing the squares as follows gives a linear diffeomorphism of $\R^{2n+2}$ that preserves the fibres of the projection
\begin{equation*}
    \R^{2n+2}\rightarrow \R^2, (x_0;\dots, x_{2n+1})\mapsto (x_0, x_{2n+1})
\end{equation*}
and that makes $h$ non degenerate on such fibres.
\begin{lem}\label{lem:hinftyGauge}
    Up to a linear gauge equivalence,
    \begin{equation*}
        h_\infty(x_0,\dots, x_{2n+1})=\frac{1}{4}\sum_{i=1}^{2n}(-1)^ix_i^2.
    \end{equation*}
\end{lem}
\begin{proof}
Let us remark that \begin{equation*}
    -x_{2i}x_{2i+1}+x_{2i+1}x_{2i+2}+x_{2n+1}x_{2i}=(x_{2i+2}-x_{2i})(x_{2i+1}-x_{2n+1})+x_{2i+2}x_{2n+1}.
\end{equation*}

We apply this identity inductively to show\begin{equation*}
    h_\infty(x_0,\dots, x_{2n+1})=\sum_{i=0}^{n-1} (x_{2i+2}-x_{2i})(x_{2i+1}-x_{2n+1}).
\end{equation*}

Every summand in the right hand side can be easily checked to be equal to
\begin{equation*}
    \frac{1}{4}\left\{(x_{2i+2}-x_{2i}+x_{2i+1}-x_{2n+1})^2-(x_{2i+2}-x_{2i}-x_{2i+1}+x_{2n+1})^2\right\}
\end{equation*}
so that 
\begin{equation*}
    h_\infty(x_0,\dots, x_{2n+1})=\frac{1}{4}\sum_{i=0}^{n-1}\left\{(x_{2i+2}-x_{2i}+x_{2i+1}-x_{2n+1})^2-(x_{2i+2}-x_{2i}-x_{2i+1}+x_{2n+1})^2\right\}.
\end{equation*}

We define the endomorphism by
\begin{align*}
    x_{2i+1}\mapsto x_{2i+2}-x_{2i}+x_{2i+1}-x_{2n+1}\\
    x_{2i+2}\mapsto x_{2i+2}-x_{2i}-x_{2i+1}+x_{2n+1}
\end{align*}
for $i$ between $0$ and $n-1$. We map $x_0$ and $x_{2n+1}$ to themselves as promised. The composition of $h$ with this linear, fibre-preserving diffeomorphism has the shape we were looking for.
\end{proof}
\begin{cor}
\label{cor:hIsGFQI}
    Up to a fibre preserving diffeomorphism generating functions of Le Calvez type are quadratic at infinity, of signature $n$ for a decomposition of length $2n+2$.
\end{cor}

\paragraph{Palais-Smale condition}

Let $\psi\in \mathrm{GL}(\R^{2n})$ be the linear map defined in the previous section, and let $g$ be the standard Euclidean product on $\R^{2n}$. Le Calvez's results (Theorem \ref{thm:LeCalLyapunov}, the existence of the Dominated Splitting and the Homothety Law (\ref{eq:homothety}) discussed in the next section) hold for the pair $(h, g)$, hence also for the pair $(h\circ \psi, \psi^*g)$, since there is a clear bijection between the flow lines of the two pairs. We are now interested in showing that $(h\circ \psi, \psi^*g)$ is Palais-Smale, to achieve the good definition of the Morse complex (up to small perturbation of the metric) on the one hand, and to preserve the Lyapunov property of the linking number on the other.

The Palais-Smale property for $(h\circ \psi, g)$ is clear because $h\circ \psi$ is quadratic at infinity. This implies what we want, by the following obvious lemma:

\begin{lem}\label{lemma:filtrChangeMetric}
    Let $f$ be a function on $\R^k$, $A\in \mathrm{GL}_{k}(\R)$, and $g$ an inner product on $\R^k$. If $(f, g)$ satisfies the Palais-Smale condition, so does $(f, A^*g)$.
\end{lem}

\begin{proof}
    The conclusion is very easy to see: given a sequence of points $x_j$ such that $\abs{f(x_j)}$ is bounded and $\norm{\nabla^{A^*g}f(x_j)}_{A^*g}\to 0$, we have to check that\begin{equation}\label{eq:nablaTo0}
        \norm{\nabla^{g}f(x_j)}_{g}\to 0
    \end{equation}
    and then apply the fact that $(f, g)$ is Palais-Smale to deduce the existence of a convergent subsequence. The limit (\ref{eq:nablaTo0}) is readily verified by standard estimates obtained using the operator norm of $A$.
\end{proof}

\begin{rem}
    The proof clearly works replacing $A$ with a diffeomorphism whose differential is bounded.
\end{rem}

In Section \ref{subsec:existenceFiltration} we are going to apply the construction from Section \ref{subsec:genFuns} to therefore define the Morse complex $CM(h, g)$, for some Riemannian metric $g$ on $\R^{2n}$.

\section{Proof of Theorem \ref{thm:A}}\label{sec:definitionFiltration}
\subsection{The dominated splitting}
Le Calvez's work \cite{leCalvez1999} provides us with another useful tool, a ``dominated splitting'' (the original wording being ``décomposition subordonnée'') of $TE$. The content of the following paragraph may be found in  \cite[Proposition 3.2.1]{leCalvez1999}.
\vspace{0.2 cm}

Let $\psi:\R\times E\rightarrow E$ be the flow of $\xi$ (the flow being complete is a result of Le Calvez). Define, for $j\in \left\{-\floor*{\frac{n}{2}}, \dots, \floor*{\frac{n}{2}}\right\}$, the subset of $T_xE=E$\begin{equation}
E_j(x)=\Set{v\in T_xE\,\vert\, \forall t\in \R,\,\,L(d_x\psi^t(v))=j}\cup \{0\}.
\end{equation}

The first result one needs to be aware of is that $E_j(x)$ is in fact a vector subspace of $E$. This is not immediately clear, since one from the definition only has invariance under scalar multiplication. Le Calvez also computes its dimension: if $n$ is even and $j=\pm\floor*{\frac{n}{2}}$ then $\dim E_j(x)=1$, in all other cases $\dim E_j(x)=2$. The $E_j$ form a decomposition of the tangent bundle in the sense that
\begin{equation}\label{eq:domDecomp}
    T_xE=\bigoplus_{j\in \left\{-\floor*{\frac{n}{2}}, \dots, \floor*{\frac{n}{2}}\right\}} E_j(x)
\end{equation}
for all $x\in E$.

The decomposition is subordinated to the flow in the sense that it is compatible with it: for all $x\in E$ and $t\in \R$ we have the equality
\begin{equation}
    d_x\psi^tE_j(x)=E_j(\psi^t(x)).
\end{equation}

We also define, for each $j$ in the image of $L$, the spaces\begin{equation}\label{eq:E+-}
    E^+_j(x)=\bigoplus_{j\leq k\leq \floor*{\frac{n}{2}}}E_k(x),\,\,\,E^-_j(x)=\bigoplus_{j\geq k\geq -\floor*{\frac{n}{2}}}E_k(x).
\end{equation}
Equivalently, $E_j^+(x)$ (resp. $E_j^-(x)$) is the set of vectors $u$ in $T_xE$ such that\begin{equation*}
    L(d_x\psi^t. u)\geq j\text{ (resp. }L(d_x\psi^t. u)\leq j\text{) }\forall t\in \R.
\end{equation*}

The last property will prove crucial in the next section, as we are going to use it to prove the existence of said filtration. 
\begin{prp}\label{prp:homLaw}

    Let $u_j\in E_j(x)$ and $u_k\in E_k(x)$ be two vectors of norm 1. We assume that $j<k$. Then, along the linearised flow of $\xi$, the norm of $u_j$ shrinks much faster than the one of $u_k$ does. More formally, there exists a $\lambda\in(0, 1)$ and a positive constant $C$ such that for all positive time $t>0$ we have:
\begin{equation}\label{eq:homothety}
\frac{\norm{d_x\psi^tu_j}}{\norm{d_x\psi^tu_k}}\leq C\lambda^t.
\end{equation}
Here the symbol $\norm{\cdot}$ indicates the norm of the standard Euclidean product on $E$, which we denoted $g$ above. We call (\ref{eq:homothety}) the ``homothety law''.
\end{prp}
This is proved in \cite[Lemma 3.2.2]{leCalvez1999}.

\begin{rem}
    When comparing with \cite{leCalvez1999}, the reader should be mindful of the fact that here $\xi$ is the negative gradient of $h$, in contrast with Le Calvez's convention. In his conventions, the function $L$ is decreasing along pairs of flow lines, and in Equation \ref{eq:homothety} one needs to swap numerator and denominator.
\end{rem}
\subsection{Existence of the filtration}\label{subsec:existenceFiltration}
The main idea we are going to exploit here is that the way the linearised flow of $\xi$ at a critical point changes the norms of unit vectors is governed by two phenomena. The first, which is classical, is simply given by the eigenvalues of the Hessian of $h$ at the critical point. The second phenomenon is the homothety law contained in Equation \ref{eq:homothety}. The two phenomena will turn out to be clearly not independent, and their interplay will yield the value that $I$ should have on the diagonal; equivalently, it will yield a well-defined notion of self-linking number of a fixed point of $\varphi$. We assume here that $\varphi$ is non degenerate (and $h$ Morse as a consequence).

We start by fixing a critical point $x$ of $h$, and we take two other points $x^1, x^2\in E$ such that\begin{equation}
    \lim_{t\to+\infty}x^i(t)=x, \,\, x=1,2
\end{equation}
and which belong to different gradient lines, i.e.\begin{equation*}
    \forall t\in \R, x^1(t)\neq x^2.
\end{equation*}
\begin{rem}
    This allows for $x^1$ or $x^2$ to be a critical point, but not both.
\end{rem}
We take the difference vector and we normalise it:
\begin{equation}
    \R\rightarrow \sphere^{2n-1}, t\mapsto v(t):=\frac{x^1(t)-x^2(t)}{\norm{x^1(t)-x^2(t)}}.
\end{equation}
We denote by $v(\infty)$ the limit-set at positive infinity:
\begin{equation}
    v(\infty)=\Set{z\in \sphere^{2n-1}\vert \exists (t_k)\subset \R, t_k\to+\infty, v(t_k)\to z}\subseteq T_xE.
\end{equation}
The first lemma shows that in fact any vector in $v(\infty)$ is tangent to the stable manifold of $\xi$ at $x$.
\begin{figure}
         \centering
        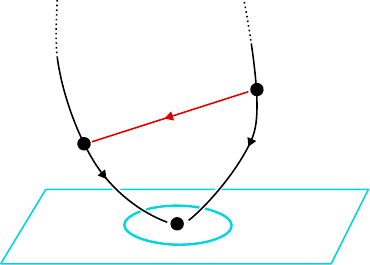
\end{figure}

\begin{lem}\label{lemma:vInfStable}
    $v(\infty)\subset T_xW^{s}(x)$.
\end{lem}
\begin{proof}
For this proof, we assume without loss of generality that $v$ admits a limit at $+\infty$ (i.e. $v(\infty)$ is a point). The fact that this is the case is going to be proven in the following lemma.

It suffices to check that the quantity \begin{equation}\label{eq:tangentSpaceHessian}
    \frac{1}{\norm{x^1(t)-x^2(t)}^2}\mathcal{H}h(x)(x^1(t)-x^2(t), x^1(t)-x^2(t))
\end{equation}
where $\mathcal{H}h(x)$ is the Hessian of $h$ at $x$, is positive at the limit. Let us consider a Morse chart centred at $x$, and assume that the metric in the chart is Euclidean: we may assume this if we allow for changes in the Riemannian metric on $E$, which we now do.

Given that $x^i(t)\to x$ as $t\to\infty$, we may assume without loss of generality that $x^i(t)$ (for $i=1, 2$) and $x^1(t)-x^2(t)$ belong in the Morse chart. Working in said chart, we are now going to prove the estimates\[
\norm{x^i(t)-x}\sim C^i( x, t)e^{-\mu_{x^i}t}, \,\,i=1,2
\]
where the $\mu_{x^i}$ are the lowest eigenvalues of $\mathcal{H}h(x)$ appearing in an expression of $x^i(t)$ in the Morse chart around $x$. To simplify the notation, let us suppose without loss of generality that $x=0$. Denote by $\phi$ the Morse chart around $x$.
\begin{align*}
    \norm{x^i(t)}_E=\frac{\norm{x^i(t)}_E}{\norm{\phi^{-1}x^i(t)}}\norm{\phi^{-1}x^i(t)}\sim\\\sim
    \frac{\norm{x^i(t)}_E}{\norm{\phi^{-1}x^i(t)}}\norm{x^i}e^{-\mu_{x^i}t}.
\end{align*}
The terms $C^i(x, t):=\frac{\norm{x^i(t)}_E}{\norm{\phi^{-1}x^i(t)}}$ are positive and bounded in $t$, as follows from a Taylor expansion to the first order of $\phi^{-1}$. They do not tend to $0$ as $t\to +\infty$ since $\phi^{-1}$ has invertible differential.
Furthermore, as $x^i(t)\to x$, both eigenvalues $\mu_{x^i}$ are positive.
\vspace{0.2 cm}

We can use this description to find a function $C^{1, 2}(\cdot)$ defined for large times, such that $\norm{x^1(t)-x^2(t)}\sim C^{1, 2}(t)e^{-\mu t}$. To do so, simply apply cosine formula, keeping in mind that the cosine of the angle spanned by the vectors $(x^1(t)-x, x^2(t)-x)$ cannot be 1 since the flow of the gradient is conjugated by a diffeomorphism to a radial vector field. If $\delta(t)$ is the cosine, we have
\begin{align*}
    \norm{x^1(t)-x^2(t)}^2=\norm{x^1(t)-x}^2+\norm{x^2(t)-x}^2-2\delta(t)\norm{x^1(t)-x}\norm{x^2(t)-x}\sim\\\sim C^1(x, t)^2e^{-2\mu_{x^1}t}+C^2(x, t)^2e^{-2\mu_{x^2}t}-2\delta(t) C^1(x, t)C^2(x, t)e^{-(\mu_{x^1}+\mu_{x^2})t}=\\=[C^1(x, t)^2+C^2(x, t)^2e^{-2(\mu_{x^2}-\mu_{x^1})t}-2\delta(t) C^1(x, t)C^2(x, t)e^{-(\mu_{x^2}-\mu_{x^1})t}]e^{-2\mu_{x^1}t}.
\end{align*}
For large $t$ \begin{equation*}
    C^1(x, t)^2+C^2(x, t)^2e^{-2(\mu_{x^2}-\mu_{x^1})t}-2\delta(t) C^1(x, t)C^2(x, t)e^{-(\mu_{x^2}-\mu_{x^1})t}>0
\end{equation*}and it does not tend\footnote{This an application of the standard inequality $a^2+b^2\geq 2ab$, where both $a$ and $b$ are positive.} to 0, as $\delta <1-\varepsilon$ for large enough $t$. 
Without loss of generality, we assume $\mu_{x^2}-\mu_{x^1}\geq 0$: the quantity above is then bounded as a function of $t$. We denote it by $C^{1,2}(t)$.

Now, let $(v_1, \dots, v_{2n})$ be a basis of $E$ in eigenvectors for $\mathcal{H}h(x)$. Let $\lambda_i$ be the eigenvalue of $v_i$. We can write in this basis the directions $x^1(t)-x$ and $ x^2(t)-x$, defining functions $\alpha_i, \beta_i:\R\rightarrow \R$.\[
\frac{x^1(t)-x}{\norm{x^1(t)-x}}=\sum_i\alpha_i(t)v_i,\,\,\,\frac{x^2(t)-x}{\norm{x^2(t)-x}}=\sum_i\beta_i(t)v_i.
\]
Using the obvious identity $x^1(t)-x^2(t)=x^1(t)-x+x-x^2(t)$, we can now expand the quantity (\ref{eq:tangentSpaceHessian}): it is thus asymptotic to \[
\frac{1}{C^{1,2}(t)^2}\sum_i\left[ C^1(x, t)\alpha_i(t)-C^2(x, t)e^{-(\mu_{x^2}-\mu_{x^1})t}\beta_i(t) \right]^2\lambda_i.
\]

Since $C(x^1, x^2, \cdot), C(x^1, x, \cdot), C(x^2, x, \cdot)$ are positive and bounded, $\mu_{x^2}-\mu_{x^1}\geq 0$, and whenever $\lambda_j<0$ we have $\alpha_j(t), \beta_j(t)\to 0$, the limit for $t\to+\infty$ of (\ref{eq:tangentSpaceHessian}) is positive or 0. Also, at least one of the terms in the sum does not vanish at the limit. If $\mu_{x^1}\neq \mu_{x^2}$, since $\alpha_i(t), \beta_i(t)$ are also bounded and $C^1(x, t)$ does not tend to 0 at infinity, it suffices to notice that $C^1(x, t)\alpha_i(t)\to 0$ for all $i$ implies that $\alpha_i(t)\to 0$ for all $i$. This is a contradiction with the fact that the $\alpha_i$ are coordinates of a unit vector, and there is at least a nonzero element in the sum: (\ref{eq:tangentSpaceHessian}) is positive in this case. If instead $\mu_{x^1}=\mu_{x^2}$, the proof is a bit more involved. We start from noticing that, the limit above being 0, it would imply
\[
\lim_{t\to+\infty}C^1(x, t)\frac{x^1(t)-x}{\norm{x^1(t)-x}}-C^2(x, t)\frac{x^2(t)-x}{\norm{x^2(t)-x}}=0
\]
which in turn yields ($\mu:=\mu_{x^1}=\mu_{x^2}>0$)

\[
\lim_{t\to+\infty}e^{\mu t}(x^1(t)-x^2(t))=0
\]
so that $x^1(t)-x^2(t)=o(e^{-\mu t})$, contradiction with the fact that $C^{1,2}(t)$ does not tend to 0. Here as well the value (\ref{eq:tangentSpaceHessian}) is then strictly positive in the limit, and we have proved the lemma.
\end{proof}

\begin{rem}
    One could see this result as a Morse-theoretical version of Siefring's description of difference of pseuholomorphic half cylinders with same asymptotics (see \cite{sie08}). We do not however need the full description he provides for our goal.
\end{rem}

A consequence of the proof is that in fact $v$ converges.
\begin{lem}
$v(\infty)$ is a point.
\end{lem}

\begin{proof}
With the notations above,\[
\frac{x^1(t)-x^2(t)}{\norm{x^1(t)-x^2(t)}}\sim\frac{1}{C^{1,2}(t)}\left(
\sum_i\alpha_i(t)C^1(x, t)v_i-\sum_i\beta_i(t)e^{-(\mu_{x^2}-\mu_{x^1})t}C^2(x, t)(t)v_i
\right).
\]
The $\alpha_i$ and the $\beta_i$ clearly admit limits as the quantities $\frac{x^1(t)-x}{\norm{x^1(t)-x}},\,\,\,\frac{x^2(t)-x}{\norm{x^2(t)-x}}$ tend to a unit vector, since they do so in a Morse chart. For the same reason, $C^1(x, t)$, $C^2(x, t)$ converge to some positive value. For $C^{1,2}(t)$, just remark that the same is true for $\delta(t)$.
\end{proof}

We shall now see how the homothety law (\ref{eq:homothety}) lets us compute $L$ in terms of the eigenvalues of the hessian matrix. We need the following obvious corollary of (\ref{eq:homothety}):

\begin{cor}\label{cor:dynamicsL}
Let $v_i$, $v_j$ be two nonzero eigenvectors of $\mathcal{H}h(x)$, of eigenvalues respectively $\mu_i$ and $\mu_j$. Then if $L(v_j)<L(v_i)$ (provided they are both defined) we have $\mu_i<\mu_j$.
\end{cor}

\begin{proof}
Without loss of generality we assume we are in a Morse chart so that, in the coordinates given by Morse lemma,
\[
\psi^t(x_1, \cdots, x_{2n})=(e^{-\mu_1 t}x_1, \cdots, e^{-\mu_{2n}t}x_{2n}).
\]
Since the flow in the chart is linear, (\ref{eq:homothety}) gives
\[
\frac{\norm{v_j^t}}{\norm{v_i^t}}=e^{-(\mu_j-\mu_i)t}\frac{\norm{v_j}}{\norm{v_i}}< C\lambda^t\to 0
\]
so that necessarily $\mu_j-\mu_i>0$. 
\end{proof}
\begin{rem}
    We have already allowed for changes in the Riemannian metric on $E$ around the critical points: we need to highlight this as it will be used several times in what follows.
\end{rem}
We want now to prove the following Lemma, which is going to later imply our main result:
\begin{lem}\label{lem:filtrationDiag}
Let $x, y, z$ be three critical points of $h$, such that there is a negative gradient line connecting $x$ to $z$, and one connecting $y$ to $z$. Then if $n$ is odd
\[
L(x- y)\leq \floor*{\frac{n}{2}}-\floor*{\frac{\mathrm{Ind}_hz}{2}},
\]
and if $n$ is even
\[
L(x - y)\leq \floor*{\frac{n}{2}}-\ceil*{\frac{\mathrm{Ind}_hz}{2}}
\]
where $\text{Ind}_hz$ is the Morse index of $h$ at $z$.
\end{lem}

\begin{proof}
    We can estimate the maximum possible  $L$ for an eigenvector in $T_xW^s(x)$ (remember that we are looking at the negative gradient flow), using the dimension of $E_j(x)$. Let us remark first that function $L$ is defined on every eigenspace: in fact the $E_j(z)$ are stabilised by the linearised flow and are therefore mapped into themselves by the Hessian. Now, the restriction to a subspace of a diagonalisable matrix is still diagonalisable: this fact together with (\ref{eq:domDecomp}) shows that there is a basis for $T_zE$ in eigenvectors of the Hessian $\mathcal{H}h(z)$, $v_1, \dots, v_{2n}$, such that $L(v_{i-1})\geq L_{i}$ for all $i$. Using the homothety law (\ref{eq:homothety}) we also know that they are ordered by eigenvalue: if $i<j$ then $\mu_i\leq\mu_j$. In particular, if $v_i$ has eigenvalue $\lambda_i$ we have the relations
\[
\lambda_0\leq \lambda_1<\dots<\lambda_{2i}\leq \lambda_{2i+1}<\dots\leq\lambda_{2n-1}
\]in the case where $n$ is odd, and
\[
\lambda_{0}<\lambda_1\leq \lambda_{2}\leq \dots < \lambda_{2i-1}\leq\lambda_{2i}<\dots < \lambda_{2n-1}
\]if $n$ is even. We can fill the bases of the different $E_j$, starting from the highest eigenvalues and from the lowest possible $L$, $-\floor*{\frac{n}{2}}$. We find via elementary calculations that if a vector $v$ is in $T_zW^s(z)$, then it is in $E_j^-(z)$ for $j= \floor*{\frac{n}{2}}-\floor*{\frac{\mathrm{Ind}_hx}{2}}$ whenever $n$ is odd (so that every $E_i$ has dimension 2), or for $j= \floor*{\frac{n}{2}}-\ceil*{\frac{\mathrm{Ind}_hx}{2}}$ if $n$ is even. 

Assume now $n$ is odd (the proof for even $n$ being identical). By Theorem \ref{thm:LeCalLyapunov}, we have the following inequalities, for $t\to+\infty$:
\begin{align*}
    \floor*{\frac{n}{2}}-\floor*{\frac{\mathrm{Ind}_hx}{2}}\geq L\left(\frac{x^1(t)-x^2(t)}{\norm{x^1(t)-x^2(t)}} \right)\geq\\\geq L\left(\frac{\psi^{-t}(x^1)-\psi^{-t}(x^2)}{\norm{\psi^{-t}(x^1)-\psi^{-t}(x^2)}} \right)\geq L(x- y)
\end{align*}
where similarly to the proof of Lemma \ref{lem:filtrationDiag} \begin{equation*}
    \lim_{t\to-\infty}x^1(t)=x,\,\,\lim_{t\to-\infty}x^2(t)=y,\,\, \lim_{t\to+\infty}x^1(t)=\lim_{t\to+\infty}x^1(t)=z.
\end{equation*}
\end{proof}
\begin{rem}
    In the proof above we have also allowed for changes in the metric around the critical points. These changes do not affect the validity of Le Calvez's Theorem \ref{thm:LeCalLyapunov}, and the inequalities above still hold.
\end{rem}

From Section \ref{subsec:genFuns} we know that $CM(h, g; \Z)$ is well defined for at least one choice of Riemannian metric $g$ on $\R^{2n}$.

We now define the function\[
I: CM_\bullet (h, g; \Z)\otimes CM_\bullet (h, g; \Z)\rightarrow \Z
\]
by
\begin{equation}\label{eq:defnFiltrMorseIndex}
I(x\otimes y)=
\begin{cases}
L(x - y) & x\neq y\\
\floor*{\frac{n}{2}}-\floor*{\frac{\mathrm{Ind}_hx}{2}} & x=y, \,\,\, n\text{ odd}\\
\floor*{\frac{n}{2}}-\ceil*{\frac{\mathrm{Ind}_hx}{2}} & x=y, \,\,\, n\text{ even}
\end{cases}
\end{equation}
on the generators and extend the usual way by
\begin{equation}\label{eq:IDefnChains}
    I\left(\sum_{i, j}\lambda_{i, j}x^i\otimes x^j\right):=\min_{i, j \vert \lambda_{i, j}\neq 0}I(x^i\otimes x^j).
\end{equation}

\begin{prp}
    The function $I$ as defined in (\ref{eq:defnFiltrMorseIndex}) induces a filtration on $CM_\bullet (h, g; \Z)\otimes CM_\bullet (h, g; \Z)$ for a choice of Riemannian metric $g$.
\end{prp}

\begin{proof}
    The differential in the tensor product is defined to be $\partial\otimes \mathrm{Id}+\varepsilon \mathrm{Id}\otimes \partial$, where $\partial$ is the Morse differential on $CM_\bullet (h, g; \Z)$ and $\varepsilon$ is a sign (standard definition of the product differential). We may apply the proofs above to the special cases in which either $x$ or $y$ (critical points which negative gradient lines of $h$ flow away from) is equal to $z$: this shows that $v(\infty)$ also in this case belongs in $T_zW^s(z)$. Switching to $-h$ one can find opposite inequalities in the case in which $x=y$, and a flow line connects $y$ to $z$. The result is $I(z, z)\geq I(x, y)$ in the former case, and $I(x,y)\leq I(x, z)$ in the latter. This proves that $I$ gives a filtration on $CM_\bullet (h, g; \Z)\otimes CM_\bullet (h, g; \Z)$.
\end{proof}
\vspace{0.2 cm}

The expression (\ref{eq:defnFiltrMorseIndex}) is rather awkward, for essentially two reasons. First, even though the actual number on the diagonal may not depend on the generating function, the description does. Second, it is not apparent what kind of topological information $I(x\otimes x)$ bears. Both points are in stark contrast with the situation outside of the diagonal of the tensor product: for $x\neq y$, $L(x-y)$ is a half of the linking number of the orbits associated to $x$ and $y$, a piece of totally intrinsic information. We aim now to decode the meaning of $I(x\otimes x)$.
\vspace{0.2 cm}

Fix $\varphi\in\Ham_c(\R^2)$ non degenerate, and $h$ as above. By Viterbo's work \cite{vit87} we know that differences of Morse indices of critical points of $h$ coincide with differences of the Conley-Zehnder indices of the associated orbits of $\varphi$. If $x\in\mathrm{Crit}(h)$, we denote by the same letter the associated fixed points of $\varphi$ associated to $x$. We may therefore normalise the Conley-Zehnder index the following way:
\begin{equation}\label{eq:CZNormalisation}
    \mathrm{Ind}_h(x)-\sigma(h)=CZ(x)+1.
\end{equation}

Recalling that $\sigma(h)=n-1$, we plug this equality and (\ref{eq:CZNormalisation}) in (\ref{eq:defnFiltrMorseIndex}). What we obtain is the more natural definition of $I$:
\begin{equation}\label{eq:defnFiltrCZ}
I(x\otimes y)=
\begin{cases}
\frac{1}{2}\mathrm{lk}(x, y) & x\neq y\\
-\ceil*{\frac{\mathrm{CZ}(x)}{2}} & x=y
\end{cases}.
\end{equation}

We are going to justify the normalisation (\ref{eq:CZNormalisation}) in the appendix, and prove that it coincides with one of the usual ones.

\subsection{Extension to GFQIs}
Fix $\varphi\in \Ham_c(\R^2)$: Laudenbach-Sikorav's Theorem \cite{sik87} provides a generating function $S: \R^2\times \R^{N}\rightarrow \R$ which is quadratic at infinity. Such a function is obtained by cutting any Hamiltonian isotopy between the identity and $\varphi$ into $\mathscr C^1$-small components, say $n$ of them. The proof of Sikorav's theorem as developed by Brunella \cite{bru91} shows that in such a case there exists a generating function defined on $\R^{4n+2}$. We are now going to prove the central theorem, let us restate it:
\begin{thm}\label{thm:filtrationAnyGFQI}
    Let $S: \R^2\times \R^k\rightarrow \R $ be a GFQI representing a non degenerate compactly supported Hamiltonian diffeomorphism $\varphi\in \Ham_c(\R^2)$. Then there exists a non degenerate quadratic form on $\R^l$ and a metric $g$ on $\R^{2+k+l}$ such that $(S\oplus Q, g)$ is a Morse-Smale, Palais-Smale pair, and the function \begin{equation*}
        I:CM(S\oplus Q, g; \Z)\otimes CM(S\oplus Q, g; \Z)\rightarrow \Z
    \end{equation*}
    \begin{equation*}
I(x\otimes y)=
\begin{cases}
\frac{1}{2}\mathrm{lk}(x,y) & x\neq y\\
-\ceil*{\frac{\mathrm{CZ}(x)}{2}} & x=y
\end{cases}
\end{equation*}
increases along the tensor product differential.
\end{thm}

\begin{proof}
Fix $S$ as in the statement, and let $h$ be a Le Calvez generating function. Then, by Viterbo Uniqueness Theorem, there exist two non degenerate quadratic forms\begin{equation*}
    Q_i: \R^{k_i}\rightarrow \R,\,\,\, i=1, 2,\,\, k_i\in \N
\end{equation*}
and a gauge equivalence 
\begin{equation*}
    \psi: \R^{2n+k_1}\rightarrow \R^{2n+k_1}
\end{equation*}
such that
\begin{equation*}
    (h\oplus Q_1)\circ \psi= S\oplus Q_2.
\end{equation*}
We may now set $Q:=Q_2$ and $l:=k_2$. Let also $g$ be the Riemannian metric on $\R^{2n}$ for which we know that $I$ defines a filtration on $CM(h, g; \Z)\otimes CM(h, g; \Z)$. If $g'$ is the associated Riemannian metric defined on $\R^{k+k_1}$ we have the following isomorphisms of complexes
\begin{align*}
    CM(h, &g; \Z)\cong CM(h\oplus Q_1, g'; \Z)\cong \\\cong CM&((h\oplus Q_1)\circ \psi, \psi^*g'; \Z)=CM(S\oplus Q, \psi^*g'; \Z).
\end{align*}
The filtration $I$ is then pushed forward along these isomorphism of chain complexes, proving the statement of the Theorem.
\end{proof}

\begin{defn}
    We say that a GFQI $S\oplus Q$ with the property spelled in the Theorem (for some Riemannian metric) bears a linking filtration.
\end{defn}

\section{Continuation maps and the proof of Theorem \ref{thm:B}}
\subsection{Braids and continuation maps}
We now analyse the behaviour of the linking filtration under continuation maps, i.e. chain homotopies between Morse complexes induced by regular homotopies between Morse data. We are going to work only with continuation maps between twist generating functions. These continuation maps are well defined by virtue of the content of Section \ref{sec:leCGF}.

Let $(h_i, g_i)$ for $i=0, 1$ be two such Morse-Smale and Palais-Smale pairs defined on the same vector bundle $E$ over $\R^2$ (in particular, the two quadratic forms at infinity coincide). From Lemma \ref{lemma:filtrChangeMetric} we know that both\begin{equation*}
    CM(h_i, g_i; \Z)\otimes CM(h_i, g_i; \Z)
\end{equation*}
carry a filtration $I$ as defined above. We want to compare the filtration on the two sides of a continuation map between the two chain complexes. We consider continuation maps of the kind we described in Section \ref{sec:leCGF}, using the setup from Section \ref{subsec:morseGFQI}. We consider a convex homotopy $(H, G)$ between $(h_0, g_0)$ and $(h_1, g_1)$, and let $s\in (-\delta, 1+\delta)$ be the time of the homotopy. The homotopy $H$ contains a term which only depends on $s$ in its construction: we ignore it, since it does not affect the properties of $L$ as long as the Riemannian metric $G$ is a small perturbation of a product metric.

Because the twist condition read off the generating function is a convex condition on the derivatives, a homotopy as above has the property that any $(H(s, \cdot)$ is a function of the kind constructed by Le Calvez, so that $L$ has the Lyapunov property for every fixed $s$, and likewise for any fixed $s$ there exists a dominated splitting as above.
\begin{lem}\label{lem:monConMaps}
Define $\tilde{L}: (-\delta, 1+\delta)\times E\rightarrow \Z$ as \begin{equation*}
    (s, x_0, \dots, x_{2n-1})\mapsto L(x_0, \dots, x_{2n-1}).
\end{equation*}
If for some $\tilde{x}\in  (-\delta, 1+\delta)\times E$ the quantity $\tilde{L}(\tilde{x})$ is not defined, then there is a positive $\varepsilon$ such that for all $0<t<\norm{\varepsilon}$, if $\tilde{x}^t=\phi^t_{-\nabla_GH}\tilde{x}$, $\tilde{L}(\tilde{x}^t)$ is defined and \begin{equation*}
    \tilde{L}(\tilde{x}^{-t})<\tilde{L}(\tilde{x}^t)
\end{equation*}for all $0<t<\varepsilon$.
\end{lem}
\begin{proof}
This Lemma is in fact a corollary of the proof of Theorem \ref{thm:LeCalLyapunov}. Assume $\tilde{L}(\tilde{x})$ is not defined: this means that if $\tilde{x}=(s, x)$, $L(x)$ is not defined either. Then by Theorem \ref{thm:LeCalLyapunov} there is an $\varepsilon>0$ such that $x^{-t}\in W_{j^-}$, $x^t\in W_{j^+}$ for all $0<t<\varepsilon$, $j^-<j^+$. Here $x^t$ is the flow of the vertical vector field $-\nabla_{g^s}h^s$ evaluated at $x$ at time $t$. Now, the $W_{j^{\pm}}$ are open in $E$: for small times then the vertical part of the flow will still be in $W_{j^-}$ in the negative direction, and in $W_{j^+}$ in the positive one.

More precisely, let $\chi^t: \R\times E\rightarrow \R$ be the first projection of the flow of $-\nabla_GH$, and $\eta^t: \R\times E\rightarrow E$ be the second one. Then by continuity of $\chi$ for any small $\delta_1$, for times $0\leq \abs{t}<\delta$, $\abs{\chi^t(s, x)-s}<\delta_1$; remark now that the transversality condition between the vector field and the boundary between $W_{j^-}$ and $W_{j^+}$ is open, so for small perturbations of the vector field $\nabla_{g^s}h^s$ it is still verified. By continuity for any $\varepsilon>0$ there is some $\delta>0$ such that for $s'\in (s-\delta, s+\delta)$l
\[
\norm{\text{pr}_2\nabla_GH(s', x)-\nabla_{g^s}h^s(x)}
<\varepsilon
\]
so by the lines above for an arbitrary small $\varepsilon$ the transversality condition is satisfied for times which are small enough (possibly smaller of course than the time in Theorem \ref{thm:LeCalLyapunov}), and we conclude.
\end{proof}

We now want to extend the result of the previous Lemma to say that $I$ increases along continuation maps. Let $(H, G)$ be a cobordism as above. The proof of Lemma \ref{lemma:vInfStable} is still valid in this context, since it is simply a dynamical result which does not depend on the form of the Morse function: if $x, y\in \text{Crit}(h_0)$, $z\in \text{Crit}(h_1)$ and there are gradient lines connecting $x$ and $y$ to $z$, with the notations above \[
v(\infty)\in T_zW^s(z)=\R\oplus E
\]
and the function $\tilde{L}$ really computes the $L$ of the vertical projection of $v(\infty)$. Applying exactly the same proof as above (it is necessary to use the dominated splitting of $E$ at $z$ for the Morse flow induced by $h_1$), we find the inequality:
\[
I(x\otimes y)\leq I(z\otimes z).
\]
We can extend this inequality to all GFQIs (up to stabilisation) pushing the filtration forward along elementary operations on both sides of the inequality.
We have thus proved the Proposition:
\begin{prp}
    Let $(S_i\oplus Q_i, g_i)$ be pairs for which the filtration $I$ is defined on\begin{equation*}
        CM(S_i\oplus Q_i, g_i; \Z)\otimes CM(S_i\oplus Q_i, g_i; \Z).
    \end{equation*}
    Assume $\Phi$ is a continuation map
    \begin{equation*}
        \Phi: CM(S_0\oplus Q_0, g_0; \Z)\rightarrow CM(S_1\oplus Q_1, g_1; \Z)
    \end{equation*}
    given by a linear cobordism $(H, G)$. Then 
    \begin{equation*}
        I(x\otimes y)\leq I(\Phi(x)\otimes \Phi(y)).
    \end{equation*}
\end{prp}

We now report a Lemma about the relation between the $\mathscr{C}^0$-metric on generating functions and the Hofer metric. This Lemma is a variation of Theorem 1.2.B in \cite{bialPol94}, and the author could not find a proof in the literature. A proof may also be given via the theory of Hamilton-Jacobi equations.

\begin{lem}\label{lem:morseConHoferLeC}
    Let $\varepsilon >0$. Given $\varphi\in\Ham_c(\R^2)$ there exists a $\delta>0$ such that for every $\psi\in\Ham_c(\R^2)$ with $d_H(\varphi, \psi)<\delta$ there are twist generating functions $h$ and $k$ for $\varphi$ and $\psi$ respectively such that
    \begin{equation}
    \norm{h-k}_\infty\leq \varepsilon.
    \end{equation}
\end{lem}
\begin{proof}
    By bi-invariance of the Hofer metric, it suffices to check that if $H$ is a compactly supported Hamiltonian generating $\varphi$, and $\norm{H}_{(1, \infty)}\leq \delta$ for a positive real number $\delta$, then there exists a twist generating function $h$ for $\varphi$ with
    \begin{equation*}
        \norm{h-h_\infty}_\infty\leq \varepsilon,
    \end{equation*}
    where $h_\infty$ was introduced in (\ref{eq:hInfty}). The strategy is the following. At first, we remark that given a $\mathscr{C}^2$-small Hamiltonian, it admits a small GFQI, without fibre variables. This fact is then applied to a subdivision of a Hamiltonian isotopy between the identity and $H$. Then, we use that, postcomposing $\mathscr{C}^2$-small Hamiltonian by a rotation, one can construct a twist generating function which is a $\mathscr{C}^0$-small deformation of the generating function of the rotation. We are going to conclude by taking the sum of the generating functions on the subdivision, as done in Section \ref{sec:leCGF}.
    
    For any compactly supported Hamiltonian $G$ on $\R^2$, consider an autonomous Hamiltonian $L$ on $\R^2$, constantly equal to $\delta$ on the support of $G$, with very small first and second derivatives, positive and compactly supported. Then $\phi^1_L$ has a generating function $F$ without fibre variables. The function $F$ is moreover positive and maximum equal to $\delta$: this is true because (in the notation from \cite{viterbo92}) $c_-(\phi^1_L)=0$ by positivity of the Hamiltonian, $c_-(\phi^1_L)=\min F$ by the fact that $F$ has no fibre variable and spectrality axiom, $c_+(\phi^1_L)=\max F$ for the same reason and every fixed point of $\phi^1_L$ has action either 0 or $\delta$.

    Take now a decomposition of $\phi^1_H$ into $N$ Hamiltonian diffeomorphisms $\phi^1_{H_i}$ for $i=0, \dots, N-1$, so that for every $t\in [0,1]$\begin{equation*}
        \phi^t_{H_i}=\phi^{\frac{1+t}{N}}_H.
    \end{equation*}
    We may thus assume $\varphi_i:=\phi^1_{H_i}$ to be generated by the Hamiltonian
    \begin{equation*}
        H_i(t, x):=\frac{1}{N}H\left(\frac{i+t}{N}, \phi^{\frac{1+t}{N}}_H(x)\right).
    \end{equation*}
    For $N$ large enough then:\begin{itemize}
        \item Each $H_i$ is $\mathscr C^2$-small;
        \item $\norm{H_i}_{(1, \infty)}\leq \delta=\norm{H}_\infty$;
        \item $\norm{H_i}_\infty\leq \delta$.
    \end{itemize}
Then we apply the above remark to each $H_i$: we obtain generating functions $F_i$ for $L_i$ which are positive and with maximum less than or equal to $\delta$. Now, since $H_i$ is $\mathscr C^2$-small, it admits a generating function $S_i$. Since $H_i\leq L_i$ we moreover have
\begin{equation*}
    S_i\leq F_i\leq \delta.
\end{equation*}
In fact $F_i-S_i$ generates $L_i-H_i$ (because $L_i$ is constant on $\mathrm{Supp}(H_i)$) which is a positive Hamiltonian, and we end as above for the inequality $S_i\leq F_i$. The inequality $F_i\leq \delta$ was already proved above.

We now construct $h_{2i+1}$, generating function of $\varphi_i\circ R^{-1}$, using $F_i$. In order to do this, we are going to suppose that the symplectic identification between $\overline{\R^2}\oplus \R^2$ and $T^*\Delta$ we use is not (\ref{eq:identCotDelta}), but rather
\begin{equation*}
    (x, y, X, Y)\mapsto (X,y, y-Y, X-x).
\end{equation*}
This choice does not affect the above argument, since the properties of spectral invariants for generating functions do not depend on the symplectic identification one chooses. With this identification, $F_i$ satisfies
\begin{equation*}
    \varphi_i(x, y)=(X, Y)\Leftrightarrow \begin{cases}
X-x=\partial_yF_i(X, y)\\
Y-y=-\partial_XF_i(X,y)
    \end{cases}.
\end{equation*}
By a simple computation therefore we see that
\begin{equation*}
    \varphi_i\circ R^{-1}(x, y)=(X, Y)\Leftrightarrow\begin{cases}
        y=-X+\partial_2F_i(X, x)\\
        Y=x-\partial_1F_i(X, x)
    \end{cases}.
\end{equation*}
Define
\begin{equation*}
    h_i(x, X):=xX-F_i(X, x).
\end{equation*}
By definition now
\begin{equation*}
    \varphi_i\circ R^{-1}(x, y)=(X, Y)\Leftrightarrow\begin{cases}
        y=-\partial_xh(x, X)\\
        Y=\partial_Xh(x, X)
    \end{cases}
\end{equation*}
which is precisely what we wanted. Recall now that the rotation $R$ has generating function $(x, X)\mapsto -xX$. Taking the sum as in Section \ref{sec:leCGF}, one defines the twist generating function
\begin{align*}
    h(x_0,\cdots, x_{2N-1})=-x_0x_1+x_1x_2-F_0(x_2, x_1)-x_2x_3+\\+\cdots-x_{2N-2}x_{2N-1}+x_{2N-1}x_0-F_{N-1}(x_0, x_{2N-1}).
\end{align*}
Since all of the $F_i$ are $\mathscr{C}^0$-small, so is the difference $h-h_\infty$, as claimed.
\end{proof}
\subsection{The proof of Theorem \ref{thm:B}}\label{sec:proofThmB}
We now adapt the arguments developed in \cite{alvMei21} using generating functions. The first step is translating in our language their Proposition 3.8.
\begin{prp}\label{prp:alvMeiPrp}
   Let $\varphi\in\Ham_c(\D)$ be non degenerate. Let $a, b\in \R\setminus\mathrm{Spec}(\varphi)$ with $ab>0$. Let $\varepsilon>0$ be such that all elements in $\mathrm{Spec}(\varphi)$ in the interval $(a-2\varepsilon,  b+2\varepsilon)$ are contained in $(a, b)$. Let $\psi\in\Ham_c(\D)$ be another non-degenerate Hamiltonian diffeomorphism, Hofer-close to $\varphi$ so that there exist twist generating functions $h$ and $k$ defined on the same space for $\varphi$ and $\psi$ with $\norm{h-k}_\infty<\varepsilon$. Then, defining homotopies of Morse data as described in Section \ref{subsec:genFuns}, for every $\varepsilon'>0$ small enough one gets continuation maps
   \begin{align}
       &G:CM^{(a, b)}(h)\rightarrow  CM^{(a+\varepsilon+\varepsilon', b+\varepsilon+\varepsilon')}(k),\\
       \hat{G}:CM&^{(a+\varepsilon+\varepsilon', b+\varepsilon+\varepsilon')}(k)\rightarrow  CM^{(a+2\varepsilon+2\varepsilon', b+2\varepsilon+2\varepsilon')}(h)\cong CM^{(a, b)}(h)
   \end{align}
   such that $\hat{G}\circ G$ is chain-homotopic to the identity of $CM^{(a, b)}(h)$.
\end{prp}
\begin{proof}
    We start by remarking that $CM^{(a+2\varepsilon+2\varepsilon', b+2\varepsilon+2\varepsilon')}(h)\cong CM^{(a, b)}(h)$ because $\varepsilon'$ is assumed to be arbitrarily small, and the condition on $\varepsilon$ is open. Moreover, given $h$, we may find $k$ as indicated above by means of an application of Lemma \ref{lem:morseConHoferLeC}. The good definition of the continuation map is standard  if we allow ourselves to perturb $\varphi$ and $\psi$ to genuinely non-degenerate Hamiltonian diffeomorphisms (this operation does not change the complexes in the actions windows we consider provided the perturbation is sufficiently small). We do so. The action windows the continuation maps are defined between follow from Corollary \ref{cor:morContActWin}. It is also standard that different homotopies of Morse data define chain-homotopic continuation maps. It is only left to check then that one can find a chain homotopy between $\hat{G}\circ G$ and the identity defined on $CM^{(a, b)}$ (in principle, gradient lines might exit this action window, obstructing the definition of the chain homotopy). Let $g:[0, 1]\times\R^2\times \R^{k}\rightarrow \R$ be the Morse function defining the chain map $\hat{G}\circ G$, and $\rho+h:[0, 1]\times\R^2\times \R^{k}\rightarrow \R$ be a small deformation of the trivial homotopy. Applying the construction in Section \ref{subsec:morseGFQI} to $g$ and $\rho+h$ one has the same estimates as in Corollary \ref{cor:morContActWin}, which imply what we aimed to prove here.
\end{proof}

We now give the definition of $\varepsilon$-isolation, to be ready to prove Alves and Meiwes's main result using generating functions.

\begin{defn}\label{defn:epsIsol}
    Let $\{x^1,\dots, x^k\}$ be a collection of critical points of a generating function $S$. We say that such collection is $\varepsilon$-isolated if the two following conditions both hold:
    \begin{itemize}
        \item For any $i, j\in\{1, \dots, k\}$, the difference $\abs{S(x^i)-S(x^j)}$ is either 0 or at least $\varepsilon$;
        \item Given any critical point $y$ of $S$, if there exists $x^i$ with $S(x^i)=S(y)$, then there exists $x^j$ such that $y=x^j$.
    \end{itemize}
\end{defn}
What follows is the proof of Theorem \ref{thm:B}.
\begin{proof}
   Let us denote by $b$ the braid $b(x^1, \dots, x^k)$, where $\{x^1, \dots, x^k\}$ is $\varepsilon$-isolated. Let $(h, g)$ be a Morse-Smale pair such that $h$ generates $\varphi$. We assume that $h$ is a twist generating function, and that points outside of the support of $\varphi$ have $h$-action 0. Define the subset of the action spectrum $\mathrm{Spec}(\underline{x})$ to be
   \begin{equation*}
   \mathrm{Spec}(\underline{x}):=\{h(x^1), \dots, h(x^k))\}.
   \end{equation*}
   For all $\kappa\in\mathrm{Spec}(\underline{x})$ we consider the Morse complexes
   \begin{equation}
   CM^{(\kappa-\varepsilon, \kappa+\varepsilon)}(h), \qquad CM^{(\kappa-2\varepsilon, \kappa+2\varepsilon)}(h).
   \end{equation}
   Due to the $3\varepsilon$-isolation requirement, the differentials of these complexes vanish. This implies that their rank equals the rank of their homologies, let this number be $n_\kappa$. It coincides with the number of orbits in $\underline{x}$ of action $\kappa$. Given $\varphi$ as in the statement of the Theorem, we take a Morse-Smale pair  $(h', g')$ generating $\varphi'$, and as per Lemma \ref{lem:morseConHoferLeC} we may assume that $h$ and $h'$  be defined on the same space, with
   \begin{equation*}
   \norm{h-h'}_{\infty}<\varepsilon.
   \end{equation*}
   
    Define continuation maps
   \begin{equation}
   G:CM^{(\kappa-\varepsilon, \kappa+\varepsilon)}(h)\rightarrow CM^{(\kappa-2\varepsilon, \kappa+2\varepsilon)}(h')
   \end{equation}
   and
   \begin{equation*}
   \hat{G}:CM^{(\kappa-2\varepsilon, \kappa+2\varepsilon)}(h')\rightarrow CM^{(\kappa-\varepsilon, \kappa+\varepsilon)}(h)
   \end{equation*}
   as in previous Lemma\footnote{Here we are using the fact that  $\underline{x}$ is $3\varepsilon$-separated. Moreover, for readability, we are suppressing the error we introduce in the definition of the continuation maps. Since the condition on $\varepsilon$ is open and that this error is arbitrarily small, it is irrelevant to the argument, therefore safe to suppress.}. Due to the vanishing of the differential of $CM^{(\kappa-\varepsilon, \kappa+\varepsilon)}(h)$, the composition
   \begin{equation*}
   \hat{G}\circ G:CM^{(\kappa-\varepsilon, \kappa+\varepsilon)}(h)\rightarrow CM^{(\kappa-\varepsilon, \kappa+\varepsilon)}(h)
   \end{equation*}
   is precisely the identity, rather than merely homotopic to it. Let now $(\sigma_l^\kappa)_{l=1}^{m_\kappa}$ be the critical points of $h'$ generating $CM^{(\kappa-2\varepsilon, \kappa+2\varepsilon)}(h')$. We say that $\sigma_l^\kappa$ appears in $G(x^i)$ if there exist a gradient line defining $G$ negatively asymptotic to $x^i$ and positively asymptotic to $\sigma_l^\kappa$. We likewise shall say that $x^i$ appears in $\hat{G}(\sigma_l^\kappa)$ is verified. We now recall the following statement, proved in the Appendix of \cite{alvMei21}:

   \textit{There exists an injective map $\mathfrak{f}_\kappa:\{1, \dots, n_\kappa\}\rightarrow\{1, \dots, m_\kappa\}$ and a bijection $\mathfrak{g}_\kappa:\{1, \dots, n_\kappa\}\rightarrow\{1, \dots, n_\kappa\}$ such that for all $i\in\{1, \dots, n_\kappa\}$, the orbit $\sigma_{\mathfrak{f}_\kappa(i)}^\kappa$ appears in $G(x^i)$ and $x^{\mathfrak{g}_\kappa(i)}$ appears in $\hat{G}(\sigma_{\mathfrak{f}_\kappa(i)}^\kappa)$.}
   
Since every permutation has finite order, there exists a positive integer $M$ such that $\mathfrak{g}_\kappa:\{1, \dots, n_\kappa\}^M$ is the identity. Let us denote by $u^1_{\kappa, i}$ a gradient line defining $G$ and connecting $x^i$ to $\sigma_{\mathfrak{f}_\kappa(i)}^\kappa$. Likewise, let $u^2_{\kappa, i}$ be a gradient line defining $\hat{G}$, connecting $\sigma_{\mathfrak{f}_\kappa(i)}^\kappa$ to $x^{\mathfrak{g}_\kappa(i)}$. For every $x^i$ in $\underline{x}$ of action $\kappa$, we consider the $2M$-tuple
\begin{equation}
(u^1_{\kappa, i}, u^2_{\kappa, i}, u^1_{\kappa, \mathfrak{g}_\kappa(i)}, u^2_{\kappa, \mathfrak{g}_\kappa(i)}, \dots, u^1_{\kappa, \mathfrak{g}^M_\kappa(i)}, u^2_{\kappa, \mathfrak{g}^M_\kappa(i)}).
\end{equation}
This is a collection of gradient lines, the first one negatively asymptotic to $x^i$ and the last one positively asymptotic to the same critical point. Now, let us consider any two points $x^i$ and $x^j$. We do not need to assume that they have the same action, say that the one of $x^i$ is $\kappa$ and the one of $x^j$ is $\eta$.

We have two associated $2M$-tuples of gradient lines (the functions $\mathfrak{g}_{\kappa}$ and $\mathfrak{g}_\eta$ may have different order, in which case we just take $M$ to be the product of the orders), which we now denote by $\underline{u}_i$ and $\underline{u}_j$. By Lemma \ref{lem:monConMaps}, we obtain a chain of inequalities:
\begin{equation}
I(x^i,  x^j)\leq I(\sigma_{\mathfrak{f}_\kappa(i)}^\kappa, \sigma_{\mathfrak{f}_\eta(j)}^\eta)\leq I(x_{\mathfrak{g}_\kappa (i)},  x_{\mathfrak{g}_\eta (j)})\leq\cdots\leq I(x_{\mathfrak{g}^M_\kappa (i)},  x_{\mathfrak{g}^M_\eta (j)})=I(x^i, x^j).
\end{equation}
In particular, $I(x^i,  x^j) = I(\sigma_{\mathfrak{f}_\kappa(i)}^\kappa, \sigma_{\mathfrak{f}_\eta(j)}^\eta)$. Since $\mathfrak{f}_\kappa$ is injective, if $x^i$ and $x^j$ have the same action $\kappa$ then $\sigma_{\mathfrak{f}_{\kappa(i)}}^\kappa\neq\sigma_{\mathfrak{f}_\kappa(j)}^\kappa$. Else , if $\kappa\neq\eta$, we know that $\sigma_{\mathfrak{f}_\kappa(i)}^\kappa\neq\sigma_{\mathfrak{f}_\eta(j)}^\eta$ because they have different actions as well. Both cases, we infer that, under Le Calvez's interpretation of points in the vector bundle where the generating functions are defined, associated to the pair of gradient lines $(u^1_{\kappa,  i}, u^1_{\eta, j})$ we have an isotopy of braids with two strands. We can repeat the argument for every pair of strands (i.e. with every pair of gradient lines): this procedure yields the braid isotopy we were looking for.
\end{proof}
\appendix
\section{Floer's picture}
Since the Morse theory of generating functions is known to be isomorphic to Floer Homology for (compactly supported) Hamiltonian diffeomorphisms (see for instance \cite{ohMilinkovic}) as persistence modules, it may come as no surprise that one may define a filtration $I$ on Floer complexes, again computing the linking number outside the diagonal. The two filtrations, as we are going to see below, coincide under the isomorphism Morse-Floer. We are going to define the Floer filtration in this section, building on work mostly by Hofer-Wysocki-Zehnder \cite{hwz98}, Hutchings \cite{hut02}, and Siefring \cite{sie11}.

\subsection{Normalisation of the Conley-Zehnder index}
The Conley-Zehnder index is going to play a fundamental role in the very definition of $I$. In order to compare our results with existing ones, we need to make sure that the normalisation we use for the index coincides with the reference literature. Remember that we had set
\begin{equation*}
     \mathrm{Ind}_h(x)-\sigma(h)=CZ(\gamma_x)+1
\end{equation*}
in (\ref{eq:CZNormalisation}). The previously cited works normalise the index the following way: if $H: \R^2\rightarrow \R$ is a $\mathscr{C}^2$-small Hamiltonian, and $\alpha$ is a 1-periodic orbit of the Hamiltonian diffeomorphism of $\R^2$ defined by $H$ and associated to a critical point $\alpha$ of $H$ (under the $\mathscr C^2$-smallness hypotheses, fixed points of the time 1-map are exactly the critical points of $H$), we set \begin{equation}\label{eq:CZNormHWZ}
    \mathrm{CZ}(\alpha)+1=\mathrm{Ind}_H(\alpha).
\end{equation}

We have to check that the two normalisations (\ref{eq:CZNormalisation}) and (\ref{eq:CZNormHWZ}) coincide on a simple example. 

Consider $H$ defined as\begin{equation*}
        H(x, y)=\frac{\varepsilon}{2}\rho(x,y)(x^2+y^2)
    \end{equation*}
    where $\rho$ is a plateau function we use to make $H$ compactly supported. For small $\varepsilon$ the Hamiltonian $H$ is indeed $\mathscr{C}^2$-small, and the origin is a non degenerate fixed point of the time 1 map $\varphi$. The Morse index of $0$ for $H$ is clearly 0, so that by (\ref{eq:CZNormHWZ})
    \begin{equation*}
        CZ(0)=-1.
    \end{equation*}
    We now have to find a generating function for $H$ of the Le Calvez kind and compute its Morse index at the point representing the origin of the plane as a fixed point. The problem is local in nature, so ignore the contribution of $\rho$, and we assume $H$ generates a genuine rotation of angle $\varepsilon$. A generating function for the symplectomorphism $R^{-1}\circ \varphi$ (remember, $R$ is the clockwise rotation of order 4) is
    \begin{equation*}
        h_0(x_0, x_1)=\frac{x_0^2}{2}\frac{\sin\varepsilon}{\cos\varepsilon}+\frac{1}{\cos\varepsilon}x_0x_1+\frac{x_1^2}{2}\frac{\sin\varepsilon}{\cos\varepsilon}
    \end{equation*}
so that using Le Calvez's results we obtain the following generating function:
\begin{equation*}
    h(x_0, x_1, x_2, x_3)=\frac{x_0^2}{2}\frac{\sin\varepsilon}{\cos\varepsilon}+\frac{1}{\cos\varepsilon}x_0x_1+\frac{x_1^2}{2}\frac{\sin\varepsilon}{\cos\varepsilon}-x_1x_2+x-2x_3-x_3x_0.
\end{equation*}
The critical point corresponding to the origin of the plane is the origin of $\R^4$: we have to compute the Morse index of $h$ there. The Hessian of $h$ at $0\in \R^4$ is
\begin{equation*}
    M=\begin{pmatrix}
       \frac{\sin\varepsilon}{\cos\varepsilon} & \frac{1}{\cos\varepsilon}& 0 &-1\\
       \frac{1}{\cos\varepsilon}& \frac{\sin\varepsilon}{\cos\varepsilon}& -1&0\\
       0& -1 & 0 & 1\\
       -1 & 0 & 1 & 0
    \end{pmatrix}
\end{equation*}
and we approximate it by means of the obvious Taylor expansion:
\begin{equation}
      M'=\begin{pmatrix}
       \frac{\varepsilon}{1-\varepsilon^2/2} & \frac{1}{1-\varepsilon^2/2}& 0 &-1\\
       \frac{1}{1-\varepsilon^2/2}& \frac{\varepsilon}{1-\varepsilon^2/2}& -1&0\\
       0& -1 & 0 & 1\\
       -1 & 0 & 1 & 0
    \end{pmatrix}.
\end{equation}
The signatures of $M$ and $M'$ coincide. We compute then the determinants of the square submatrices of $M'$ whose diagonal is contained in the one of $M'$: for small $\varepsilon>0$, the top-left entry 
\begin{equation*}
    \frac{\varepsilon}{1-\varepsilon^2/2}
\end{equation*}
is positive, and the determinants of the other submatrices are negative. There is exactly one sign change for these determinants, which implies\begin{equation*}
    \sigma(M')=\sigma(M')=1.
\end{equation*}
By definition then
\begin{equation*}
    \mathrm{Ind}_h(0)=1.
\end{equation*}
From (\ref{eq:CZNormalisation}) and knowing that $\sigma(h)=1$ we find
\begin{equation*}
    CZ(\gamma_0)=-1
\end{equation*}
which is what we wanted to show.

\subsection{A filtration on the Floer complex}
In this part of the Appendix we shall adopt the notation contained in \cite{wendl20}.
Consider $\varphi\in\Ham_c(\Sigma)$ generated by a Hamiltonian $H$. If $\varphi$ is non degenerate, for a generic choice of almost complex structure $J$ the Floer complex $(CF(H, J; \Z), \partial)$ is well defined. Let $x, y\in \Fix(\varphi)$. We may define the function

\begin{equation}
I(x\otimes y)=
\begin{cases}
\frac{1}{2}\mathrm{lk}(x , y) & x\neq y\\
-\ceil*{\frac{\mathrm{CZ}(x)}{2}} & x=y
\end{cases}
\end{equation}
on the product of generators in complete analogy to what was done above, and we extend to the whole complex $CF(H, J; \Z)\otimes CF(H, J; \Z)$ as done in (\ref{eq:IDefnChains}). In this case it induces a filtration as well: we have to prove that
\begin{prp}\label{prp:IIsFiltrFloer}
\begin{equation*}
    I(\partial(x\otimes y))\geq I(x\otimes y).
\end{equation*}
In this last expression, $\partial$ denotes the tensor product differential
\begin{equation*}
    \partial(x\otimes y):=\partial x\otimes y+(-1)^{CZ(x)}x\otimes \partial y.
\end{equation*}
\end{prp}

Before starting with the proof of the proposition, we mention the following equality:
\begin{lem}
If $x$ is a 1-periodic orbit of $\varphi$,
\begin{equation}
I(x\otimes x)=\alpha^\tau_-(x)
\end{equation}
for a trivialisation $\tau$, canonical up to homotopy.
\end{lem}
\begin{proof}
First off, $\tau$ is constructed taking a capping disc $u$ in $\Sigma$ for the orbit $x$. Since the pullback bundle $u^*T\Sigma$ has contractible basis, it is globally trivial and any two trivialisations are homotopic. Adding to this fact our assumption that $\pi_2(\Sigma)=0$, we obtain that $\tau$ is indeed canonical up to homotopy. After remarking this the Lemma is a trivial consequence of\begin{equation}
    -CZ^\tau(\gamma)=2\alpha^\tau_-(\gamma)+p(\gamma)=2\alpha^\tau_+(\gamma)-p(\gamma),
\end{equation} because $p$ takes values in $\{0, 1\}$.
\end{proof}
We are going to use the canonical trivialisation $\tau$ of the above Lemma throughout this section.
\begin{proof} \textit{Of Proposition\ref{prp:IIsFiltrFloer}} In the Hamiltonian Floer setting, we are lead to count intersections between pseudoholomorphic curves defining the differential, in order to estimate the variation of $I$. Let $u, v: \R_s\times\sphere^1_t\rightarrow \Sigma$ be two Floer cylinders with distinct images. The two cylinders have therefore only finitely many intersections: since the images are distinct, the asymptotic description given by Siefring in \cite{sie08} implies that the cylinders do not intersect in a neighbourhood of infinity, but they cannot have infinitely many intersections in a compact set, else they would have the same image. Let $u(\pm\infty, \cdot)=x_{\pm}$, $v(\pm\infty, \cdot)=y_{\pm}$. Here we make the abuse of notation of giving the same name to a fixed point and the periodic orbit through that fixed point. Denote by $\bar u$ and $\bar v$ the graphs of $u$ and $v$ respectively. Remark that we are allowing for the case in which the cylinders are trivially constant at a Hamiltonian orbit.

Looking at the (unordered) pair $(u, v)$ as a homotopy between the braids composed by the asymptotic Hamiltonian orbits, we easily see from (\ref{eq:lkHomotopyIntersections}) that to each (transverse) intersection counted in $\bar u\cdot\bar v$ corresponds a linking increase of 2 since such intersections are positive. If $x_{\pm}\neq y_\pm$ then we have the following formula:
\begin{equation*}
    \bar u*\bar v=\bar u\cdot\bar  v +\iota_{\infty}(\bar u, \bar v)=\bar u\cdot \bar v = I(x_+\otimes  y_+)-I(x_-\otimes y_-).
\end{equation*}
In the second equality we use $x_{\pm}\neq y_{\pm}$ to deduce that $\iota_\infty(\bar u, \bar v)=0$, and in the third equality we use this assumption again to infer that
\begin{equation*}
 I(x_\pm\otimes y_\pm)=\frac{1}{2}\mathrm{lk}(x_\pm, y_\pm).
\end{equation*}

Assume now that $x_-=y_-$ (the case at $+\infty$ being analogous). Then the same calculations as above lead to
\begin{equation*}
    \bar u*\bar v=\bar u\cdot \bar v + \iota^\tau_{-\infty}(\bar u, \bar v)-\alpha^\tau_+(x_-)=I(x_+, y_+)-I(x_-, y_-)
    \end{equation*}
    since $\bar u\cdot\bar v+ \iota^\tau_{-\infty}(\bar u, \bar v)=\frac{1}{2}\mathrm{lk}(x_+, y_+)$ by (\ref{eq:lkHomotopyIntersections}). These are all the possible cases we have to consider: in the tensor product differential we have to count constant cylinders at periodic orbits and Floer cylinders connecting periodic orbits of Conley-Zehnder index difference 1.
\end{proof}
\section{Collections of points as braids for general GFQIs}\label{sec:pointsEBraidsGFQI}
In the Section \ref{sec:leCGF}, we defined a function $\gamma$ taking as input a point of the vector bundle where a twist generating function $h$ is defined, and yielding a loop in $\R^2$. We now aim to do the same for a general GFQI, applying Viterbo's Uniqueness Theorem for GFQI. Suppose that $S:\R^{l}\rightarrow \R$ is a GFQI bearing a linking filtration. Then $S$ is obtained from a twist generating function $h: E=\R^{2n}\rightarrow \R$ via (Stabilisation) and (Gauge equivalence). To define a map
\begin{equation*}
    \gamma: \R^{l}\times [0, 2n]\rightarrow \R^2.
\end{equation*}
it is enough to see how the function $\gamma$ defined for twist generating functions may be pushed-forward along elementary operations.

Let $h: E=\R^{2n}\rightarrow \R$ be a twist generating function for a Hamiltonian diffeomorphism $\varphi$ associated to a decomposition
\begin{equation*}
    \varphi=\varphi_{n-1}\circ\cdots\varphi_0.
\end{equation*}
Let $\gamma:E\times[0, 2n]\rightarrow\R^2$ be the function described in Section \ref{sec:pointsEBraidsLeC}, associating to a point in $E$ a loop in the plane. Applying the (Shift) operation clearly does not affect $E$, and therefore leaves $\gamma$ unchanged. Let us stabilise $h$ by a non-degenerate quadratic form $Q:\R^k\rightarrow\R$. We have to define a function $\gamma^{h\oplus Q}$ for
\begin{equation*}
    h\oplus Q: E\times \R^k\rightarrow\R.
\end{equation*}
The natural candidate is clearly
\begin{equation}
    \gamma^{h\oplus Q}:E\times \R^k\times [0, 2n+k]\rightarrow \R^2
\end{equation}
defined as follows. If $(x, v)\in E\times\R^k$,
\begin{equation}
    \gamma^{h\oplus Q}(x,  v, t)=\gamma(x, t)\text{ for }t\in [0, 2n],\,\,\gamma^{h\oplus Q}(x,  v, t)=\gamma(x, 2n)\text{ for }t\in [2n, 2n+k].
\end{equation}
This is a good definition in the following sense: if we replace the decomposition of $\varphi$ we used to define $h$ by
\begin{equation*}
    \varphi=\mathrm{Id}^k\circ \varphi_{n-1}\circ \cdots\circ \varphi_0
\end{equation*}
the new twist generating function we obtain is a stabilisation\footnote{Not all stabilisations may be obtained this way. It is easy to see, in fact, that the dimension of the vector bundle we sum here is always even, and the space where the quadratic form is positive definite has the same dimension as that where it is negative-definite.} of $h$. The function $\gamma$ of the new function $h$ satisfies the definition we just gave, as one can see from the description provided in Lemma \ref{lem:braidTypesCoincide}.

\begin{rem}
    As one sees, for $v\neq v'\in \R^k$ one has, for all $x\in E$,
    \begin{equation*}
        \gamma^{h\oplus Q}(x, v, \cdot)=\gamma^{h\oplus Q}(x, v', \cdot).
    \end{equation*}
    However, the map $(x, v)\mapsto \gamma^{h\oplus Q}(x, v, \cdot)$ is injective when restricted to $E\times\{0\}$.
\end{rem}

Treating the case of the operation (Gauge Equivalence) is not very different. Let $h, \gamma$ be as above, and $\Psi:E\rightarrow E$ be a fibre-preserving diffeomorphism. We define 
\begin{equation}
    \gamma^{h\circ \Psi}:E\times[0, 2n]\rightarrow\R^2
\end{equation}
by a simple push-forward:
\begin{equation}
    \gamma^{h\circ \Psi}(x, t)=\gamma(\Psi^{-1}(x),t).
\end{equation}

Given then a GFQI $S$ such that there exist non-degenerate quadratic forms $Q, q$, a gauge equivalence $\Psi$ and a twist generating function satisfying
\begin{equation*}
    S\oplus Q=(h\oplus q)\circ\Psi
\end{equation*}
the function $\gamma^{S\oplus Q}$ for $S\oplus Q$ is defined as
\begin{equation*}
    \gamma^{S\oplus Q}(x, t)=\gamma^q(\Psi^{-1}x, t).
\end{equation*}
\printbibliography
\end{document}

%% file: preamble.tex
\usepackage{tikz-cd}
\usepackage{tikz}
\usepackage{bbm}

\usepackage{dsfont}
\usepackage{url}
\usepackage{hyperref}
\usepackage{amsmath}
\usepackage{amssymb}
\usepackage{mathrsfs}
\usepackage{braket}
\usepackage{mathtools}
\usepackage{booktabs}
\usepackage{caption}
\usepackage{cancel}
\usepackage{epstopdf}
\usepackage{comment}
\usepackage{appendix}
\DeclarePairedDelimiter{\abs}{\lvert}{\rvert}
\DeclarePairedDelimiter{\norm}{\lVert}{\rVert}
\newcommand{\numberset}{\mathbb}
\newcommand{\N}{\numberset{N}}
\newcommand{\R}{\numberset{R}}

\newcommand{\Z}{\numberset{Z}}

\newcommand{\sphere}{\mathbb{S}}

\newcommand{\Ham}{\text{Ham}}

\newcommand{\D}{\mathbb{D}}

\usepackage{amsthm}
\newtheorem{thm}{Theorem}[section]

\newtheorem{lem}[thm]{Lemma}
\newtheorem{prp}[thm]{Proposition}
\newtheorem{cor}[thm]{Corollary}
\newtheorem{defn}[thm]{Definition}
\newtheorem{nota}[thm]{Notation}

\newtheorem{rem}[thm]{Remark}
\newtheorem*{es}{Example}

\DeclareMathOperator{\Fix}{Fix}

\DeclarePairedDelimiter\ceil{\lceil}{\rceil}
\DeclarePairedDelimiter\floor{\lfloor}{\rfloor}

\makeatletter
\newcommand*{\@old@slash}{}\let\@old@slash\slash
\def\slash{\relax\ifmmode\delimiter"502F30E\mathopen{}\else\@old@slash\fi}
\makeatother

\makeatletter
\def\bign#1{\mathclose{\hbox{$\left#1\vbox to8.5\p@{}\right.\n@space$}}\mathopen{}}
\def\Bign#1{\mathclose{\hbox{$\left#1\vbox to11.5\p@{}\right.\n@space$}}\mathopen{}}
\def\biggn#1{\mathclose{\hbox{$\left#1\vbox to14.5\p@{}\right.\n@space$}}\mathopen{}}
\def\Biggn#1{\mathclose{\hbox{$\left#1\vbox to17.5\p@{}\right.\n@space$}}\mathopen{}}
\makeatother

%% file: definitionFiltration.pdf_tex
\begingroup%
  \makeatletter%
  \providecommand\color[2][]{%
    \errmessage{(Inkscape) Color is used for the text in Inkscape, but the package 'color.sty' is not loaded}%
    \renewcommand\color[2][]{}%
  }%
  \providecommand\transparent[1]{%
    \errmessage{(Inkscape) Transparency is used (non-zero) for the text in Inkscape, but the package 'transparent.sty' is not loaded}%
    \renewcommand\transparent[1]{}%
  }%
  \providecommand\rotatebox[2]{#2}%
  \newcommand*\fsize{\dimexpr\f@size pt\relax}%
  \newcommand*\lineheight[1]{\fontsize{\fsize}{#1\fsize}\selectfont}%
  \ifx\svgwidth\undefined%
    \setlength{\unitlength}{177.15084286bp}%
    \ifx\svgscale\undefined%
      \relax%
    \else%
      \setlength{\unitlength}{\unitlength * \real{\svgscale}}%
    \fi%
  \else%
    \setlength{\unitlength}{\svgwidth}%
  \fi%
  \global\let\svgwidth\undefined%
  \global\let\svgscale\undefined%
  \makeatother%
  \begin{picture}(1,0.71695613)%
    \lineheight{1}%
    \setlength\tabcolsep{0pt}%
    \put(0,0){\includegraphics[width=\unitlength,page=1]{definitionFiltration.pdf}}%
    \put(0.28063218,0.46115018){\color[rgb]{0,0,0}\makebox(0,0)[lt]{\lineheight{1.25}\smash{\begin{tabular}[t]{l}$x^1(t)-x^2(t)$\end{tabular}}}}%
    \put(0.06029292,0.27240669){\color[rgb]{0,0,0}\makebox(0,0)[lt]{\lineheight{1.25}\smash{\begin{tabular}[t]{l}$x^1(t)$\end{tabular}}}}%
    \put(0.71299618,0.45532991){\color[rgb]{0,0,0}\makebox(0,0)[lt]{\lineheight{1.25}\smash{\begin{tabular}[t]{l}$x^2(t)$\\\end{tabular}}}}%
    \put(0.70218707,0.32229481){\color[rgb]{0,0,0}\makebox(0,0)[lt]{\lineheight{1.25}\smash{\begin{tabular}[t]{l}$\xi$\end{tabular}}}}%
    \put(0.2972616,0.26076615){\color[rgb]{0,0,0}\makebox(0,0)[lt]{\lineheight{1.25}\smash{\begin{tabular}[t]{l}$\xi$\end{tabular}}}}%
    \put(0.74126607,0.16098985){\color[rgb]{0,0,0}\makebox(0,0)[lt]{\lineheight{1.25}\smash{\begin{tabular}[t]{l}$T_xWs(x)$\end{tabular}}}}%
    \put(0,0){\includegraphics[width=\unitlength,page=2]{definitionFiltration.pdf}}%
    \put(0.40098521,0.06558494){\color[rgb]{0,0,0}\makebox(0,0)[lt]{\lineheight{1.25}\smash{\begin{tabular}[t]{l}$v(\infty)$\end{tabular}}}}%
  \end{picture}%
\endgroup%

%% file: bibliography.bib
@article{ohMilinkovic,
author = {Milinkovi\'c, Darko and Oh, Yong-Geun},
year = {1997},
month = {01},
pages = {},
title = {Floer homology as the stable Morse homology},
volume = {34},
journal = {Journal of the Korean Mathematical Society}
}

@article{viterbo92,
author = "Viterbo, Claude",
year =  "1992",
title= "Symplectic topology as the geometry of generating functions",
journal = "Mathematische Annalen",
doi = "10.1007/BF01444643"
}

@Article{leCalvez1999,
 Author = {Patrice {Le Calvez}},
 Title = {{D\'ecomposition des diff\'eomorphismes du tore en application d\'eviant la verticale (avec un appendice en collaboration avec Jean-Marc Gambaudo)}},
 FJournal = {{M\'emoires de la Soci\'et\'e Math\'ematique de France. Nouvelle S\'erie}},
 Journal = {{M\'em. Soc. Math. Fr., Nouv. S\'er.}},
 ISSN = {0249-633X; 2275-3230/e},
 Volume = {79},
 Pages = {148},
 Year = {1999},
 Publisher = {Soci\'et\'e Math\'ematique de France (SMF), Paris},
 Language = {French},
 MSC2010 = {37E30 37-02 37C25 37E40 37E45 37J45 37J50},
 Zbl = {1039.37023}
}

@Article{hlrs2016,
 Author = {Vincent {Humili\`ere} and Fr\'ed\'eric {Le Roux} and Sobhan {Seyfaddini}},
 Title = {{Towards a dynamical interpretation of Hamiltonian spectral invariants on surfaces}},
 FJournal = {{Geometry \& Topology}},
 Journal = {{Geom. Topol.}},
 ISSN = {1465-3060},
 Volume = {20},
 Number = {4},
 Pages = {2253--2334},
 Year = {2016},
 Publisher = {Mathematical Sciences Publishers (MSP), Berkeley, CA; Geometry \& Topology Publications c/o University of Warwick, Mathematics Institute, Coventry},
 Language = {English},
 DOI = {10.2140/gt.2016.20.2253},
 MSC2010 = {53D40 37E30},
 Zbl = {1356.53082}
}

@Book{leCalvez91,
 Author = {Patrice {Le Calvez}},
 Title = {{Propri\'et\'es dynamiques des diff\'eomorphismes de l'anneau et du tore}},
 FJournal = {{Ast\'erisque}},
 Journal = {{Ast\'erisque}},
 ISSN = {0303-1179},
 Volume = {204},
 Pages = {131},
 Year = {1991},
 Publisher = {Paris: Soci\'et\'e Math\'ematique de France},
 Language = {French},
 MSC2010 = {37B99 37C70 54H20 39B12},
 Zbl = {0784.58033}
}

@Article{vit87,
 Author = {Claude {Viterbo}},
 Title = {{Intersection de sous-vari\'et\'es lagrangiennes, fonctionnelles d'action et indice des syst\`emes hamiltoniens. (Intersection of Lagrangian submanifolds, action functionals and indices of Hamiltonian systems)}},
 FJournal = {{Bulletin de la Soci\'et\'e Math\'ematique de France}},
 Journal = {{Bull. Soc. Math. Fr.}},
 ISSN = {0037-9484},
 Volume = {115},
 Pages = {361--390},
 Year = {1987},
 Publisher = {Soci\'et\'e Math\'ematique de France (SMF), Paris},
 Language = {French},
 DOI = {10.24033/bsmf.2082},
 MSC2010 = {37J99},
 Zbl = {0639.58018}
}

@Article{conn21,
 Author = {Connery-Grigg, Dustin},
 Title = {Hamiltonian {Floer} theory on surfaces: linking, positively transverse foliations and spectral invariants},
 FJournal = {Inventiones Mathematicae},
 Journal = {Invent. Math.},
 ISSN = {0020-9910},
 Volume = {237},
 Number = {3},
 Pages = {1377--1468},
 Year = {2024},
 Language = {English},
 DOI = {10.1007/s00222-024-01274-0},
 zbMATH = {7895066}
}

@Article{hwz98,
 Author = {Hofer, H. and Wysocki, K. and Zehnder, E.},
 Title = {Properties of pseudoholomorphic curves in symplectisations. {I}: {Asymptotics}},
 FJournal = {Annales de l'Institut Henri Poincar{\'e}. Analyse Non Lin{\'e}aire},
 Journal = {Ann. Inst. Henri Poincar{\'e}, Anal. Non Lin{\'e}aire},
 ISSN = {0294-1449},
 Volume = {13},
 Number = {3},
 Pages = {337--379},
 Year = {1996},
 Language = {English},
 DOI = {10.1016/S0294-1449(16)30108-1},
 zbMATH = {939891},
 Zbl = {0861.58018}
}

@Article{alvMei21,
 Author = {Alves, Marcelo R. R. and Meiwes, Matthias},
 Title = {Braid stability and the {Hofer} metric},
 FJournal = {Annales Henri Lebesgue},
 Journal = {Ann. Henri Lebesgue},
 ISSN = {2644-9463},
 Volume = {7},
 Pages = {521--581},
 Year = {2024},
 Language = {English},
 DOI = {10.5802/ahl.205},
 zbMATH = {7914800}
}

@Article{art25,
 Author = {Artin, E.},
 Title = {Theorie der {Z{\"o}pfe}.},
 FJournal = {Abhandlungen aus dem Mathematischen Seminar der Universit{\"a}t Hamburg},
 Journal = {Abh. Math. Semin. Univ. Hamb.},
 ISSN = {0025-5858},
 Volume = {4},
 Pages = {47--72},
 Year = {1925},
 Language = {German},
 DOI = {10.1007/BF02950718},
 zbMATH = {2592684},
 JFM = {51.0450.01}
}

@Book{kasTur08,
 Author = {Kassel, Christian and Turaev, Vladimir},
 Title = {Braid groups. {With} the graphical assistance of {Olivier} {Dodane}.},
 FSeries = {Graduate Texts in Mathematics},
 Series = {Grad. Texts Math.},
 ISSN = {0072-5285},
 Volume = {247},
 ISBN = {978-0-387-33841-5},
 Year = {2008},
 Publisher = {New York, NY: Springer},
 Language = {English},
 DOI = {10.1007/978-0-387-68548-9},
 zbMATH = {5268073},
 Zbl = {1208.20041}
}

@Article{gon11,
 Author = {Gonz{\'a}lez-Meneses, Juan},
 Title = {Basic results on braid groups.},
 FJournal = {Annales Math{\'e}matiques Blaise Pascal},
 Journal = {Ann. Math. Blaise Pascal},
 ISSN = {1259-1734},
 Volume = {18},
 Number = {1},
 Pages = {15--59},
 Year = {2011},
 Language = {English},
 DOI = {10.5802/ambp.293},
 zbMATH = {5903953},
 Zbl = {1264.20035}
}

@Article{sie11,
 Author = {Siefring, Richard},
 Title = {Intersection theory of punctured pseudoholomorphic curves},
 FJournal = {Geometry \& Topology},
 Journal = {Geom. Topol.},
 ISSN = {1465-3060},
 Volume = {15},
 Number = {4},
 Pages = {2351--2457},
 Year = {2011},
 Language = {English},
 DOI = {10.2140/gt.2011.15.2351},
 zbMATH = {5996767},
 Zbl = {1246.32028}
}

@Book{wendl20,
 Author = {Wendl, Chris},
 Title = {Lectures on contact 3-manifolds, holomorphic curves and intersection theory},
 FSeries = {Cambridge Tracts in Mathematics},
 Series = {Camb. Tracts Math.},
 ISSN = {0950-6284},
 Volume = {220},

 Year = {2020},
 Publisher = {Cambridge: Cambridge University Press},
 Language = {English},
 DOI = {10.1017/9781108608954},
 zbMATH = {7162755},
 Zbl = {1490.53002}
}

@Book{sib04,
 Author = {Siburg, Karl Friedrich},
 Title = {The principle of least action in geometry and dynamics},
 FSeries = {Lecture Notes in Mathematics},
 Series = {Lect. Notes Math.},
 ISSN = {0075-8434},
 Volume = {1844},
 ISBN = {3-540-21944-7},
 Year = {2004},
 Publisher = {Berlin: Springer},
 Language = {English},
 DOI = {10.1007/b97327},
 zbMATH = {2122041},
 Zbl = {1060.37048}
}

@Article{the99,
 Author = {Th{\'e}ret, David},
 Title = {A complete proof of {Viterbo}'s uniqueness theorem on generating functions},
 FJournal = {Topology and its Applications},
 Journal = {Topology Appl.},
 ISSN = {0166-8641},
 Volume = {96},
 Number = {3},
 Pages = {249--266},
 Year = {1999},
 Language = {English},
 DOI = {10.1016/S0166-8641(98)00049-2},
 zbMATH = {1340387},
 Zbl = {0952.53037}
}

@Article{sik87,
 Author = {Sikorav, J.-C.},
 Title = {Probl{\`e}mes d'intersections et de points fixes en g{\'e}om{\'e}trie hamiltonienne. ({Intersection} and fixed-point problems in {Hamiltonian} geometry)},
 FJournal = {Commentarii Mathematici Helvetici},
 Journal = {Comment. Math. Helv.},
 ISSN = {0010-2571},
 Volume = {62},
 Pages = {62--73},
 Year = {1987},
 Language = {French},
 DOI = {10.1007/BF02564438},
 zbMATH = {4121063},
 Zbl = {0684.58015}
}

@Book{mcDSal16,
 Author = {McDuff, Dusa and Salamon, Dietmar},
 Title = {Introduction to symplectic topology},
 Edition = {3rd edition},
 FSeries = {Oxford Graduate Texts in Mathematics},
 Series = {Oxf. Grad. Texts Math.},
 Volume = {27},

 Year = {2016},
 Publisher = {Oxford: Oxford University Press},
 Language = {English},
 zbMATH = {6638013},
 Zbl = {1380.53003}
}

@Article{bru91,
 Author = {Brunella, Marco},
 Title = {On a theorem of {Sikorav}},
 FJournal = {L'Enseignement Math{\'e}matique. 2e S{\'e}rie},
 Journal = {Enseign. Math. (2)},
 ISSN = {0013-8584},
 Volume = {37},
 Number = {1-2},
 Pages = {83--87},
 Year = {1991},
 Language = {English},
 zbMATH = {24965},
 Zbl = {0739.58015}
}

@Article{hut02,
 Author = {Hutchings, Michael},
 Title = {An index inequality for embedded pseudoholomorphic curves in symplectizations},
 FJournal = {Journal of the European Mathematical Society (JEMS)},
 Journal = {J. Eur. Math. Soc. (JEMS)},
 ISSN = {1435-9855},
 Volume = {4},
 Number = {4},
 Pages = {313--361},
 Year = {2002},
 Language = {English},
 DOI = {10.1007/s100970100041},
 zbMATH = {1903703},
 Zbl = {1017.58005}
}

@Article{sie08,
 Author = {Siefring, Richard},
 Title = {Relative asymptotic behavior of pseudoholomorphic half-cylinders},
 FJournal = {Communications on Pure and Applied Mathematics},
 Journal = {Commun. Pure Appl. Math.},
 ISSN = {0010-3640},
 Volume = {61},
 Number = {12},
 Pages = {1631--1684},
 Year = {2008},
 Language = {English},
 DOI = {10.1002/cpa.20224},
 zbMATH = {5368947},
 Zbl = {1158.53068}
}

@Article{bialPol94,
 Author = {Bialy, Misha and Polterovich, Leonid},
 Title = {Geodesics of {Hofer}'s metric on the group of {Hamiltonian} diffeomorphisms},
 FJournal = {Duke Mathematical Journal},
 Journal = {Duke Math. J.},
 ISSN = {0012-7094},
 Volume = {76},
 Number = {1},
 Pages = {273--292},
 Year = {1994},
 Language = {English},
 DOI = {10.1215/S0012-7094-94-07609-6},
 zbMATH = {727825},
 Zbl = {0819.58006}
}

@Article{hut16,
 Author = {Hutchings, Michael},
 Title = {Mean action and the {Calabi} invariant},
 FJournal = {Journal of Modern Dynamics},
 Journal = {J. Mod. Dyn.},
 ISSN = {1930-5311},
 Volume = {10},
 Pages = {511--539},
 Year = {2016},
 Language = {English},
 DOI = {10.3934/jmd.2016.10.511},
 zbMATH = {6990500},
 Zbl = {1402.37066}
}

@misc{nelWei23,
      title={Torus knotted Reeb dynamics and the Calabi invariant}, 
      author={Jo Nelson and Morgan Weiler},
      year={2023},
      eprint={2310.18307},
      archivePrefix={arXiv},
      primaryClass={math.GT}
}

@misc{nelWei24,
      title={Torus knot filtered embedded contact homology of the tight contact 3-sphere}, 
      author={Jo Nelson and Morgan Weiler},
      year={2024},
      eprint={2306.02125},
      archivePrefix={arXiv},
      primaryClass={math.GT}
}

@Article{tra94,
 Author = {Traynor, Lisa},
 Title = {Symplectic homology via generating functions},
 FJournal = {Geometric and Functional Analysis. GAFA},
 Journal = {Geom. Funct. Anal.},
 ISSN = {1016-443X},
 Volume = {4},
 Number = {6},
 Pages = {718--748},
 Year = {1994},
 Language = {English},
 DOI = {10.1007/BF01896659},
 zbMATH = {718645},
 Zbl = {0822.58020}
}

@Article{leC23,
 Author = {Le Calvez, Patrice},
 Title = {A finite dimensional proof of a result of {Hutchings} about irrational pseudo-rotations},
 FJournal = {Journal de l'{\'E}cole Polytechnique -- Math{\'e}matiques},
 Journal = {J. {\'E}c. Polytech., Math.},
 ISSN = {2429-7100},
 Volume = {10},
 Pages = {837--866},
 Year = {2023},
 Language = {English},
 DOI = {10.5802/jep.234},
 zbMATH = {7690440},
 Zbl = {1526.37050}
}

@misc{braPir24,
      title={Spectral invariants for non-compactly supported Hamiltonians on the disc, and an application to the mean action spectrum}, 
      author={Barney Bramham and Abror Pirnapasov},
      year={2024},
      eprint={2403.07863},
      archivePrefix={arXiv},
      primaryClass={math.SG}
}
